\documentclass[fleqn,reqno,11pt,a4paper,final]{amsart}

\usepackage[a4paper,left=20mm,right=20mm,top=20mm,bottom=20mm,marginpar=20mm]{geometry}
\usepackage[english]{babel}
\usepackage{microtype}
\usepackage{amsmath}
\usepackage{amssymb}
\usepackage{amsthm}
\usepackage{amscd}
\usepackage[utf8]{inputenc}
\usepackage{color}
\usepackage[english=american]{csquotes}
\usepackage[final]{graphicx}
\usepackage{hyperref}
\usepackage[capitalise, noabbrev]{cleveref}
\usepackage{calc}
\usepackage{enumitem}
\setlist[enumerate]{label =(\arabic*), ref=\arabic*, font=\normalfont}
\usepackage{mathptmx}
\usepackage{mathtools}
\usepackage{textcmds} 
\usepackage{tikz}
\usetikzlibrary{shapes} 
\usepackage{graphicx}
\usepackage{float}
\usepackage{stix}
\usepackage{enumitem}
\usepackage{dsfont}

\usepackage[backend=biber,bibencoding=utf8,firstinits,maxnames=4,style=alphabetic,isbn=false, doi=false, eprint= true, url= false,date=year]{biblatex}
\usepackage[english=american]{csquotes}
\binoppenalty=\maxdimen
\relpenalty=\maxdimen

\graphicspath{{../Pictures/}}

\makeatletter
\renewcommand\paragraph{\@startsection{paragraph}{4}{\z@}
  {1ex \@plus1ex \@minus.2ex}
  {-1em}
  {\normalfont\normalsize\bfseries}}
\makeatother

\numberwithin{equation}{section}

\newtheoremstyle{thmlemcorr}{10pt}{10pt}{\itshape}{}{\bfseries}{.}{10pt}{{\thmname{#1}\thmnumber{ #2}\thmnote{ (#3)}}}
\newtheoremstyle{thmlemcorr*}{10pt}{10pt}{\itshape}{}{\bfseries}{.}\newline{{\thmname{#1}\thmnumber{ #2}\thmnote{ (#3)}}}
\newtheoremstyle{remexample}{10pt}{10pt}{}{}{\bfseries}{.}{10pt}{{\thmname{#1}\thmnumber{ #2}\thmnote{ (#3)}}}
\newtheoremstyle{ass}{10pt}{10pt}{}{}{\bfseries}{.}{10pt}{{\thmname{#1}\thmnumber{ A#2}\thmnote{ (#3)}}}

\theoremstyle{thmlemcorr}
\newtheorem{theorem}{Theorem}
\numberwithin{theorem}{section}
\newtheorem{lemma}[theorem]{Lemma}

\newtheorem{proposition}[theorem]{Proposition}

\theoremstyle{thmlemcorr*}
\newtheorem*{theorem*}{Theorem}
\newtheorem{lemma*}[theorem]{Lemma}
\newtheorem{corollary*}[theorem]{Corollary}
\newtheorem{proposition*}[theorem]{Proposition}
\newtheorem{problem*}[theorem]{Problem}
\newtheorem{conjecture*}[theorem]{Conjecture}
\newtheorem{definition*}[theorem]{Definition}
\newtheorem{assumption*}[theorem]{Assumption}

\theoremstyle{remexample}
\newtheorem{remark}[theorem]{Remark}

\theoremstyle{ass}

\newcommand{\Brm}{\mathrm{B}}
\newcommand{\Crm}{\mathrm{C}}

\newcommand{\Hrm}{\mathrm{H}}

\newcommand{\Lrm}{\mathrm{L}}

\newcommand{\Wrm}{\mathrm{W}}

\newcommand{\Ccal}{\mathcal{C}}

\newcommand{\Fcal}{\mathcal{F}}

\newcommand{\Kcal}{\mathcal{K}}
\newcommand{\Lcal}{\mathcal{L}}

\newcommand{\Ocal}{\mathcal{O}}

\newcommand{\Scal}{\mathcal{S}}

\newcommand{\Ucal}{\mathcal{U}}

\newcommand{\Xcal}{\mathcal{X}}
\newcommand{\Ycal}{\mathcal{Y}}

\newcommand{\bfk}{\mathbf{k}}
\newcommand{\bfe}{\mathbf{e}}
\newcommand{\bfx}{\mathbf{x}}

\newcommand{\bfp}{\mathbf{p}}
\renewcommand{\k}{\mathbf{k}}
\newcommand{\x}{\mathbf{x}}

\newcommand{\Nbb}{\mathbb{N}}

\newcommand{\Rbb}{\mathbb{R}}

\newcommand{\Zbb}{\mathbb{Z}}

\newcommand{\de}{\, \mathrm{d}}

\DeclareMathOperator{\dist}{dist}

\DeclareMathOperator{\ran}{ran}

\newcommand{\ee}{\mathrm{e}}
\newcommand{\ii}{\mathrm{i}}
\newcommand{\set}[2]{\left\{\, #1 \ \ \textup{\textbf{:}}\ \ #2 \,\right\}}

\newcommand{\norm}[1]{\|#1\|}
\newcommand{\normlr}[1]{\left\|#1\right\|}

\newcommand{\abs}[1]{|#1|}
\newcommand{\abslr}[1]{\left|#1\right|}

\newcommand{\sprlr}[1]{\left( #1 \right)}

\newcommand{\dprlr}[1]{\left\langle #1 \right\rangle}

\newcommand{\dd}{\;\mathrm{d}}

\newcommand{\N}{\mathbb{N}}
\newcommand{\R}{\mathbb{R}}

\newcommand{\Z}{\mathbb{Z}}

\newcommand{\loc}{\mathrm{loc}}

\newcommand{\per}{\mathrm{per}}

\DeclareMathOperator{\Id}{Id}

\def\XXint#1#2#3{{\setbox0=\hbox{$#1{#2#3}{\int}$}
		\vcenter{\hbox{$#2#3$}}\kern-.5\wd0}}

\renewcommand{\phi}{\varphi}

\definecolor{tumblue}{RGB}{0, 101, 189}

\begin{document}
	
	\title[]{Pattern formation and film rupture in a two-dimensional thermocapillary thin-film model of the Bénard--Marangoni problem}
    
	\author{Stefano Böhmer}
	\address{\textit{Stefano Böhmer:} Centre for Mathematical Sciences, Lund University, P.O. Box 118, 221 00 Lund, Sweden}
	\email{stefano.bohmer@math.lu.se}
	
	\author{Bastian Hilder}
    \address{\textit{Bastian Hilder:}  Department of Mathematics, Technische Universität München, Boltzmannstraße 3, 85748 Garching b.\ München, Germany}
    \email{bastian.hilder@tum.de}
    
    \author{Jonas Jansen}
    \address{\textit{Jonas Jansen:}   Fraunhofer Institute for Algorithms and Scientific Computing SCAI, Schloss Birlinghoven, 53757 Sankt Augustin, Germany}
    \email{jonas.jansen@scai.fraunhofer.de}
	
	\begin{abstract}
		We study two-dimensional, stationary square and hexagonal patterns in the thermocapillary deformational thin-film model for the fluid height $h$
        \begin{equation*}
    		\partial_t h+
    		\nabla\cdot
    		\sprlr{
    			h^3
    			\sprlr{\nabla\Delta h-g\nabla h}
    			+
    			M\frac{h^2}{(1+h)^2}\nabla h
    		}
    		=0
    		,\quad t>0
    		,\quad x\in\R^2,
	    \end{equation*}
        that can be formally derived from the Bénard–Marangoni problem via a long-wave approximation. Using a linear stability analysis, we show that the flat surface profile corresponding to the pure conduction state destabilises at a critical Marangoni number $M^*$ via a conserved long-wave instability. For any fixed absolute wave number $k_0$, we find that square and hexagonal patterns bifurcate from the flat surface profile at $M=M^* + 4k_0^2$. Using analytic global bifurcation theory, we show that the local bifurcation curves can be extended to global curves of square and hexagonal patterns with constant absolute wave number and mass. 
        We exclude that the global bifurcation curves are closed loops through a global bifurcation in cones argument, which also establishes nodal properties for the solutions.
        Furthermore, assuming that the Marangoni number is uniformly bounded on the bifurcation branch, we prove that solutions exhibit film rupture, that is, their minimal height tends to zero. This assumption is substantiated by numerical experiments.
    \end{abstract}
	\vspace{4pt}
	
	\maketitle
	
    
	
	\section{Introduction}
    \label{sec:intro}
    Convection in thin fluids driven by thermal gradients leads to a rich variety of nonlinear phenomena. One of the first experimental observations was the formation of hexagonal patterns, which was made in 1900 by Henri Bénard \cite{benard1900,benard1901}, who considered a thin viscous fluid placed on top of a heated plane surface. A first attempt at a theoretical explanation was made by Lord Rayleigh \cite{rayleigh1916}, who conjectured that the phenomena are driven by buoyancy effects. However, Block \cite{block1956} and Pearson \cite{pearson1958} noticed that for sufficiently thin films as considered by Bénard, the phenomenon is driven rather by surface tension gradients caused by temperature fluctuations along the surface, called \emph{thermocapillary effects}.
    Note that similar phenomena driven by surface tension gradients have also been noted by Thomson \cite{thomson1882}. 
    
    Later experimental observations in \cite{schatz1995,vanhook1995,vanhook1997} also observed that film rupture and dewetting, where the fluid surface touches the bottom solid with a receding contact line, can occur. Since these phenomena are characterised by a large deformation of the surface, they are typically referred to as deformational instability.

    To understand this deformational behaviour, we consider an asymptotic thin-film model for the height $h = h(t,x,y)$ of a viscous Newtonian fluid resting on a heated impermeable plane, given by
	\begin{equation}
		\label{eq:thin-film}
		\partial_t h+
		\nabla\cdot
		\sprlr{
			h^3
			\sprlr{\nabla\Delta h-g\nabla h}
			+
			M\frac{h^2}{(1+h)^2}\nabla h
		}
		=0
		,\quad t>0
		,\quad x\in\R^2.
	\end{equation}
    Here, $g>0$ is a gravitational constant and $M>0$ is the Marangoni number, which is proportional to the difference between the temperature of the solid and the temperature of the ambient gas.
    This model can be formally derived from the full Boussinesq–Navier–Stokes equations using a long-wave multiscale expansion and lubrication approximation; see \cite{nepomnyashchy2012,nazaripoor2018,bruell2024}. Note that in the derivation, a linear dependency of surface tension on the temperature is assumed.

    The thin-film equation \eqref{eq:thin-film} has recently been studied in one spatial dimension in \cite{bruell2024}. Here, it was found that the pure conduction state $\bar{h} \equiv 1$ of a flat surface undergoes a conserved long-wave instability as $M$ increases. The authors prove a global bifurcation of spatially periodic solutions with an arbitrary fixed period from the pure conduction state using the Marangoni number $M$ as the bifurcation parameter. Along this bifurcation curve, solutions accumulate at a weak, stationary periodic film-rupture solution, where the surface profile vanishes at discrete, periodic points. The analysis in \cite{bruell2024} is based on a planar Hamiltonian structure of the spatial dynamics formulation of the bifurcation problem. Therefore, the analysis is restricted to one spatial dimension, and the two-dimensional case remained open.
    
	In this paper, we extend the global bifurcation results to the two-dimensional case. In our analysis, we study stationary solutions with square or hexagonal periodicity and symmetry. As in the one-dimensional case, we prove that for any arbitrary, fixed period, there is a global bifurcation branch of periodic solutions and that, under an additional assumption, the solutions approach film rupture. However, the existence of weak, stationary film-rupture solutions found in the one-dimensional case remains an open problem, see also \cref{sec:discussion} for a detailed discussion.

    \subsection{Main results of the paper}

    We summarise the main results and techniques of this paper. We prove that

    \begin{itemize}
        \item the pure conduction state $\bar{h} \equiv 1$ undergoes a conserved long-wave instability at the critical Marangoni number $M^* = 4g$;
        \item for every absolute wave number $k_0 > 0$, there is a critical Marangoni number $M^*(k_0) = 4g + 4k_0^2$, where spatially periodic solutions with either $D_4$- or $D_6$-symmetry bifurcate from the pure conduction state, see \cref{thm loc bif squ,thm loc bif hex};
        \item the local bifurcation curves can be extended to global curves using analytic global bifurcation theory, see \cref{thm glob bif};
        \item the global bifurcation curves do not return to the trivial branch, see \cref{prop not loop} and, assuming that $M(s)$ is uniformly bounded along the bifurcation branch, the minimum of $h(s)$ approaches $0$ along the branch, see \cref{prop cond rupture}.
    \end{itemize}
    Additionally, we perform a numerical continuation of the bifurcation branches using \textsc{pde2path} \cite{dohnal}, to confirm the assumption that $M(s)$ is bounded along the bifurcation branch and to investigate secondary bifurcations, see Section \ref{sec:numerics}. 
    The results are illustrated in \cref{fig:num bif diag squ joined,fig:num bif diag hex joined}.

    \begin{figure}[h!]
		\centering
        \includegraphics[width=0.8\textwidth]{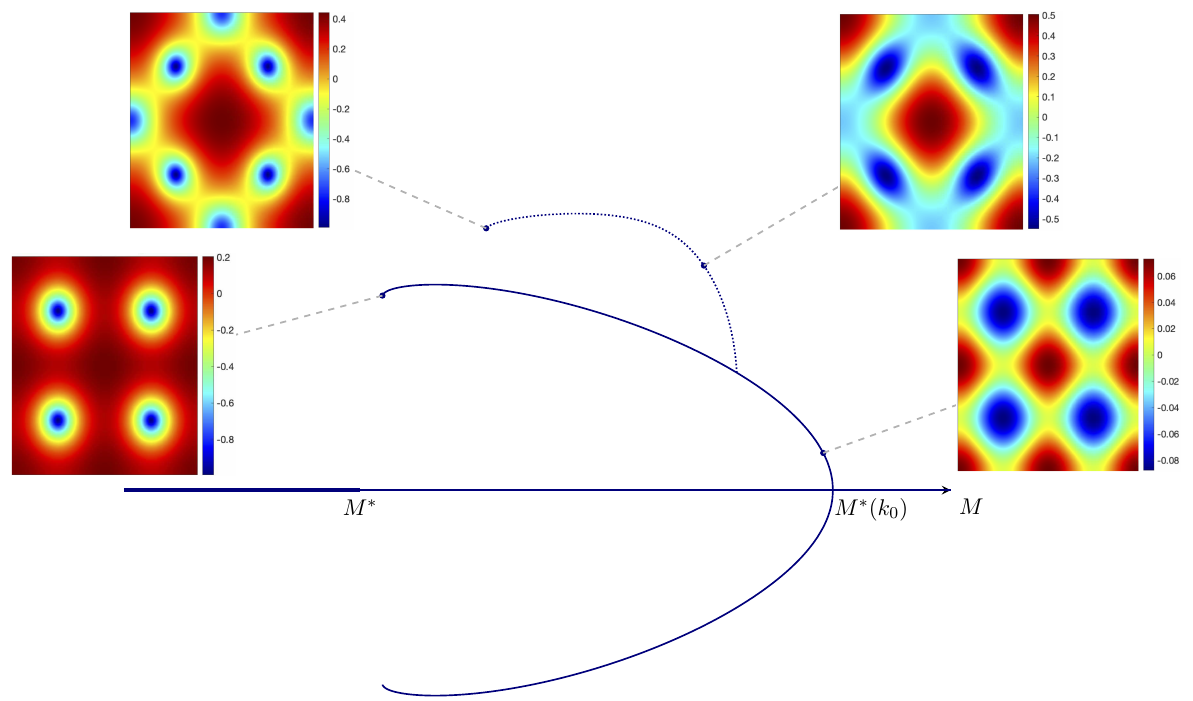}
		\caption{Square patterns bifurcate via a pitchfork bifurcation from the pure conduction state. Furthermore, the numerical experiments suggest the existence of a period-doubling secondary bifurcation. For a detailed description of the numerical continuation, see \cref{sec:numerics}.}
		\label{fig:num bif diag squ joined}
	\end{figure}

    \begin{figure}[h!]
		\centering
        \includegraphics[width=0.8\textwidth]{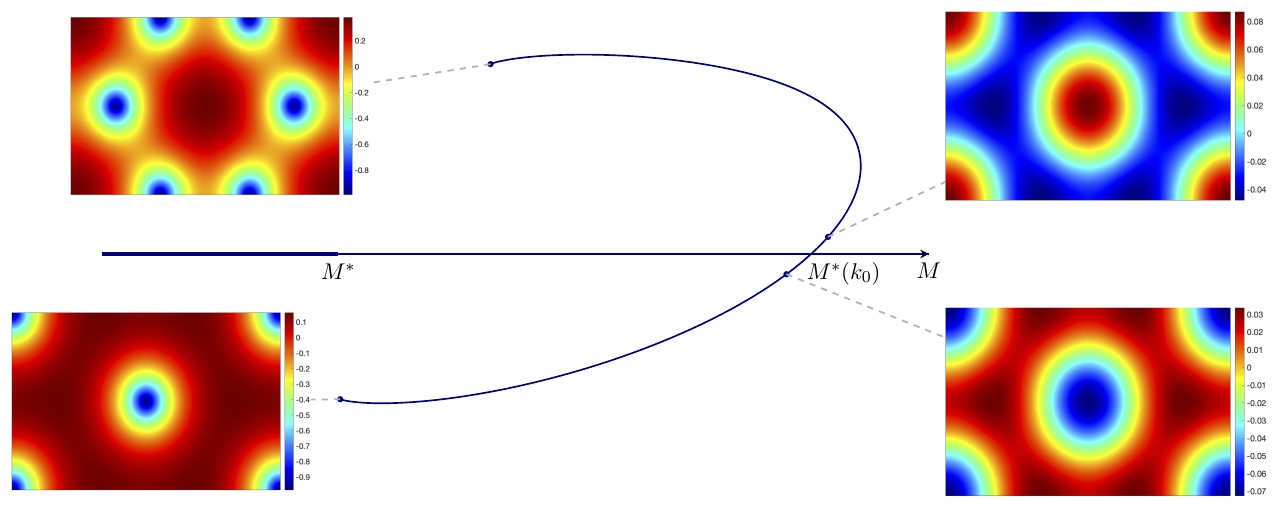}
		\caption{Hexagonal patterns bifurcate transcritically from the pure conduction state. The upper branch consists of up-hexagons and the lower branch of down-hexagons. For a detailed description of the numerical continuation, see \cref{sec:numerics}.}
		\label{fig:num bif diag hex joined}
	\end{figure}

    \paragraph{Global bifurcation analysis}
    The (conserved) long-wave instability of the pure conduction state $\bar{h} \equiv 1$ in \eqref{eq:thin-film}, which is established in \cref{sec: lin stab subsec}, indicates that stationary periodic solutions bifurcate from the pure conduction state $\bar{h}\equiv1$ for $M>M^\ast = 4g$.
    Motivated by numerical observations, cf.~\cite{oron2000}, we are interested in square and hexagonal periodic patterns, which are symmetric with respect to the respective dihedral groups, $D_4$ and $D_6$.
    
    We prove the existence of such solutions in \cref{thm loc bif squ,thm loc bif hex} as follows: first, we fix a square or hexagonal fundamental domain by fixing $k_0>0$, see \cref{fig:pattern_size}.
    \begin{figure}[h!]
	\centering
    \includegraphics[width=0.3\textwidth]{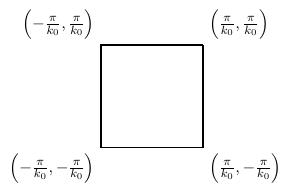}\hspace{2cm}
    \includegraphics[width=0.3\textwidth]{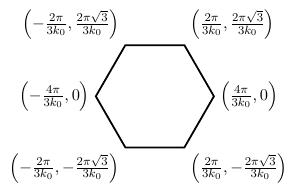}
	\caption{For a fixed absolute wave number $k_0$, we fix a square (left) and hexagonal domain (right). The sizes are chosen such that the dual lattice is generated by wave vectors with length $k_0$.}
	\label{fig:pattern_size}
    \end{figure}
    Thus, we expect solutions which fit into the fixed fundamental domain to bifurcate at $M=M^\ast(nk_0)=4g+4(nk_0)^2$, where $n\in\Nbb_+$.
    We reformulate the existence problem for stationary and periodic solutions to \eqref{eq:thin-film} as a non-local, second-order elliptic problem for $v\equiv h-1$, which reads as
    \begin{equation}\label{eq:elliptic-problem-intro}
        \Delta v - gv + M \left(\dfrac{1}{2+v} + \log\left(\dfrac{1+v}{2+v}\right)\right) =  M K(v),
    \end{equation}
    where $K(v) \in \R$ depends non-locally on $v$ and is given in \cref{eq:K}. The constant $K(v)$ is chosen to guarantee that $v$ has vanishing mean along the bifurcation branch. 
    The elliptic equation \eqref{eq:elliptic-problem-intro} is then formulated as the bifurcation problem \eqref{bif problem} with the Marangoni number $M$ as bifurcation parameter.
    Due to the restriction to periodic solutions with square or hexagonal symmetry, this problem is amenable to the standard local bifurcation theory by Crandall and Rabinowitz about bifurcations from a simple eigenvalue, see \cite{buffoni2003}. 
    
    In \cref{expansions squ,expansions hex}, we compute the expansions of the respective local bifurcation branch $\sprlr{v(s),M(s)}$ about the bifurcation point, which allow us to characterise the types of bifurcation. Similar to roll waves in one dimension, square patterns arise through a subcritical pitchfork bifurcation. 
    In contrast, hexagonal patterns emanate from a transcritical bifurcation, and this qualitative difference is caused by resonance in the hexagonal Fourier lattice, that is, quadratic combinations of critical Fourier modes can again lie at a critical Fourier mode.
    
    We also study some stability properties of the small-amplitude patterns on the local bifurcation branches.
    Using abstract results on spectral theory for perturbed operators by Kato \cite{kato1995}, we prove in \cref{sec:spec-stability} spectral instability with respect to both $\Lrm^2$-localised and subharmonic perturbations.
    In contrast, we obtain spectral stability with respect to superharmonic perturbations.
    Co-periodic perturbations with full $D_4$ or $D_6$ symmetry are covered in \cref{sec:stability-coperiodic} using a centre-manifold analysis. 
    It turns out that square patterns and down-hexagons are nonlinearly unstable with respect to co-periodic, symmetric perturbations, whereas up-hexagons are nonlinearly stable.

    In \cref{thm glob bif}, we then extend the local bifurcation branches of both, square and hexagonal patterns to global ones using the analytic global bifurcation theorem, see \cite{constantin2016a}.
    This additionally provides three alternatives for the qualitative behaviour of the global branch, see \cref{thm glob bif}. The bifurcation curve is a closed loop, or the solutions become unbounded along the bifurcation curve, or there is a subsequence of solutions approaching the boundary of the phase space. As explained below, the last alternative constitutes film rupture.
    
    We investigate these alternatives further to narrow them down. We exclude the alternative that the branch is a closed loop in \cref{prop not loop} using a theorem on global bifurcation in cones, see \cite{buffoni2003}, and the ellipticity of \eqref{bif problem}.
    A direct consequence of the proof is that we establish a nodal property, as we prove that extrema are invariant along the bifurcation branch, see \cref{remark nodal}.
    Next, regarding the norm blow-up alternative, in \Cref{prop no derivative blowup}, we exclude a mere blow-up of derivatives along the bifurcation branch by using a uniform regularity estimate. That is, solutions along the bifurcation branch cannot develop singular profiles containing kinks or corners unless they approach the boundary.
    Finally, we prove a conditional result on the occurrence of the third alternative, which corresponds to film rupture, as we discuss in the following.
    
    \paragraph{Conditional film rupture}
    We are able to prove the following conditional result on film rupture of the square- or hexagon-periodic solutions to the thin-film equation \eqref{eq:thin-film} in \cref{prop cond rupture}: assuming that the Marangoni number $M(s)$ is uniformly bounded on the global bifurcation branch, there is a subsequence of solutions $v(s_j)$ on the branch exhibiting film rupture in the sense that $\min_\Omega v(s_j)\to -1$ as $j\to\infty$.
    Recall that the film height is retrieved via $h\equiv1+v$.
    
    This conditional result is justified by two observations:
    first, for the one-dimensional problem, a bound on $M(s)$ is established in \cite{bruell2024} using a reformulation of the bifurcation problem as a planar Hamiltonian system.
    Second, we analyse the two-dimensional case numerically in \cref{sec:numerics} and all experiments indicate that, indeed, $M(s)$ admits a uniform bound.
    
    We refer to \cref{sec:discussion} for a more detailed discussion of the lack of a rigorous bound on $M(s)$.
    There, we also discuss that even upon assuming that such a bound exists, convergence to a weak film-rupture solution, as established in \cite{bruell2024} in the one-dimensional case, remains an open problem due to worse Sobolev embeddings in two dimensions.
    
    \paragraph{Numerical analysis}
    We complement our investigation of the global bifurcation of square and hexagonal patterns with a numerical analysis of the bifurcation problem \eqref{bif problem}.
    Using the \textsc{pde2path} library for \textsc{Matlab}, see \cite{dohnal,uecker2014,uecker2021}, we perform a numerical continuation of branches of solutions to the bifurcation problem \eqref{bif problem}, see \cref{sec:numerics}.
    Our numerical results match our analytical results in \cref{thm loc bif squ,thm loc bif hex,thm glob bif}. In particular, we detect bifurcation branches of square and hexagon solutions.
    The numerical continuation of the branch of square patterns is of subcritical pitchfork-type, whereas in the case of hexagon patterns, the numerical continuation branch is transcritical.
    Moreover, the numerical solutions feature the nodal property along the branch, see \cref{remark nodal}.
    
    All numerical continuations terminate close to film rupture, in the sense that the minimum of the numerical solution is close to, but still larger than $-1$.
    As noted above, the bifurcation parameter $M(s)$ seems to admit a uniform bound, as it eventually decreases in all cases.
    
    It is worth mentioning that numerically, we also detect secondary bifurcations in the case of square patterns.
    The solutions on the secondary bifurcation branches feature a period that is an integer multiple of the period of solutions on the primary branch. 
    
    \subsection{Related results}

    In this paper, we study film rupture and dewetting caused by thermocapillary effects. Besides experimental observations in \cite{vanhook1997}, these phenomena were also studied numerically, see \cite{nazaripoor2018} for Newtonian thin films and \cite{mohammadtabar2022} for non-Newtonian power-law thin films. Film rupture and dewetting have also been studied in thin-film models, where the dynamics are driven by Van der Waals forces instead of thermal gradients. We refer to \cite{sharma1998,sharma1999} for numerical results on two-dimensional pattern formation and film rupture and \cite{witelski2000,vaynblat2001} for an analysis of line and radially symmetric solutions using self-similar dynamics, as well as the recent review paper \cite{witelski2020} on one-dimensional dynamics. Although the physical driving forces are different, the observed effects seem similar. In particular, in both settings, the flat surface destabilises via a conserved long-wave instability. We also refer to \cite{shklyaev2017a} for a recent overview of long-wave instabilities in fluids.

    Experimental studies similar to Bénard's \cite{benard1900,benard1901} observe a competition between regular surface patterns and film rupture, see e.g.~\cite{vanhook1997}. While both phenomena have been observed in experiments, the thin film model \eqref{eq:thin-film}, which was derived for perfect heat conduction at the fluid-solid interface \cite{oron1997a}, seems to capture film rupture but no observable small amplitude pattern close to the flat surface. 
    However, under the assumption of low thermal conductivity on the fluid-solid interface, Shklyaev et.~al.~\cite{shklyaev2012} derived a different asymptotic model from the Bénard–Marangoni problem and showed that this model exhibits a conserved short-wave instability which typically leads to the bifurcation of small-amplitude patterns. This observation was recently made rigorous by the second and third author in \cite{hilder2024a}, where they establish a rigorous description of small-amplitude square and hexagonal patterns bifurcating from the flat surface as well as the spatial transition between different patterns.

    As discussed above, we use a global bifurcation approach to study large-amplitude steady solutions to the thin-film equation \eqref{eq:thin-film}. Notably, \cref{eq:thin-film} can be rewritten as
    \begin{equation}
        \partial_t h + \nabla \cdot \left(h^3 \nabla\left(\Delta h - g h + M \left(\dfrac{1}{1+h} + \log\left(\dfrac{h}{1+h}\right)\right)\right)\right)=0,
    \end{equation}
    which is a generalised Cahn–Hilliard equation, see also \cite{shklyaev2017a}. A global bifurcation result for the standard two-dimensional Cahn–Hilliard equation with a double-well potential has been obtained in \cite{kielhofer1997}. The nodal structure of the solutions can be established similarly to our result, see \cref{remark nodal} and \cite{healey1991}. In addition, a detailed global structure of the bifurcation curves can be established. However, we point out that the proof of the global properties found in \cite{kielhofer1997} relies on the potential having a double-well structure. We also refer to \cite{maier-paape1997} for similar results on global bifurcation for the Cahn–Hilliard equation with triangular symmetry. This paper also generalises previous results obtained in \cite{fife1997} on the robustness of nodal structures on the local bifurcation branch with respect to perturbations of the underlying equation. Finally, we mention results on spinodal decomposition in the Cahn–Hilliard equation \cite{maier-paape1998,maier-paape2000}, which refers to the observation that solutions of the Cahn–Hilliard equation first form a spatial structure with finite wavelength. This is followed by a coarsening phase, where spatial structures with a longer wavelength emerge. Notably, this phenomenon can also be observed in thermocapillary thin films, \cite{bestehorn2003}, as well as thin films on hydrophobic surfaces \cite{witelski2020}, see also \cite{bestehorn2009}.

    In the one-dimensional case \cite{bruell2024}, the Sivashinsky equation \cite{sivashinsky1983} has been derived as an amplitude equation capturing the dynamics close to the onset of the conserved long-wave instability at $M = M^*$ using a multiscale expansion. It turns out that the same computation also yields the Sivashinsky equation
    \begin{equation*}
        \partial_t V  = - \Delta V  - \Delta^2 V -2g \nabla\cdot(V\nabla V)
    \end{equation*}
    as an amplitude equation in the two-dimensional case, see e.g.~\cite{shklyaev2017a}. Note that this is a Cahn–Hilliard equation with quadratic nonlinearity, and thus has a cubic potential which is not bounded from below. Therefore, it is expected that finite-time blow-up occurs \cite{shklyaev2017a} for an open class of initial conditions. For a related equation, this was confirmed in \cite{bernoff1995} for periodic initial conditions with sufficiently long wavelength.
        
    \subsection{Outline of the paper}
    The paper is structured as follows: 
    In the first part of \cref{sec:setup}, we carry out a linear stability analysis of the fourth-order equation \labelcref{eq:thin-film}, which features a long-wave instability for $M>M^\ast$.
    Thereafter, we introduce the notions of $\Omega$-symmetry and $\Omega$-periodicity for a fundamental domain $\Omega$, and we introduce the respective function spaces.
    Lastly, we reduce the stationary problem to second order.
    
    \cref{sec loc bif} contains the local bifurcation analysis: First, we prove the existence of local branches of square patterns and compute a local expansion of the branch at the bifurcation point.
    Thereafter, we cover local branches of hexagonal patterns.
    Lastly, we analyse the stability of small-amplitude patterns in both cases.

    In \cref{sec glob bif}, the local bifurcation branches are continued to global ones, and we proceed with a qualitative analysis of these branches. In particular, this includes the nodal properties and our conditional result on film rupture.

    \cref{sec:numerics} is devoted to a numerical analysis of the bifurcation problem \eqref{bif problem}. 
    This includes a numerical justification of the conditional result.

    Lastly, in \cref{sec:discussion} we discuss open problems.
	
	\section{Linear stability analysis, stationary problem, and setup}
    \label{sec:setup}
    
    The goal of this section is threefold: First, we perform a linear stability analysis of the pure conduction state in \cref{eq:thin-film} and identify a critical Marangoni number $M^* > 0$, where a long-wave instability occurs. Afterwards, we introduce the notion of periodic and symmetric functions with respect to square and hexagonal domains.
    Finally, we reduce the stationary problem to a nonlocal second-order elliptic problem.

    \subsection{Linear stability analysis}
    \label{sec: lin stab subsec}
	We begin with a linear spectral analysis.
    \Cref{eq:thin-film} can be written as $h_t-\Fcal_M(h)=0$, where $\Fcal_M$ is a nonlinear degenerate-elliptic partial differential operator in space.
    The linearisation of $\Fcal_M$ about the pure conduction state $\Bar{h}=1$ is
	\begin{equation}
		\label{linearised operator}
		\Lcal_M v=
		-\Delta^2v+\sprlr{g-\frac{M}{4}}\Delta v.
	\end{equation}
	Inserting the ansatz $v(t,\bfx)=\ee^{\lambda(\bfk)t-\ii\bfk\cdot\bfx}$, we compute the \emph{dispersion relation}
	\begin{equation}
        \label{disp rel}
		\lambda(\bfk)=\hat{\Lcal}_M(\bfk)=-\abs{\bfk}^4-\sprlr{g-\frac{M}{4}}\abs{\bfk}^2
		,\quad \bfk=(k_1,k_2)\in\Zbb^2,
	\end{equation}
    which is illustrated in \cref{fig:instability}.
	The dispersion relation is radially symmetric due to the rotational invariance of \Cref{eq:thin-film}. Additionally, as long as $M \leq 4g$, the dispersion relation is non-positive and $\k = 0$ is its only root, and we define $M^\ast\coloneqq 4g$. If $M = M^*(k_0) := M^* + 4k_0^2$, the dispersion relation satisfies $\lambda_{M^*(k_0)}(\bfk)>0$ if and only if $|\k| \in (0,k_0)$. Therefore, the pure conduction state undergoes a conserved long-wave instability, or sideband instability, at $M = M^*$. In particular, the pure conduction state is linearly unstable with respect to periodic perturbations with wave number $|\k| \in (0,k_0)$. This indicates the bifurcation of periodic solutions for Marangoni numbers $M>M^*$.

    \begin{figure}[H]
    \centering
    \includegraphics[width=0.5\textwidth]{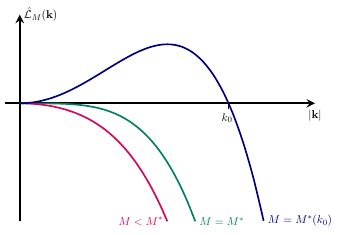}
    \caption{The dispersion relation $\hat{\Lcal}_M(\k)$ for the different parameter regimes: for $M < M^*$, the pure conduction state is stable. At $M=M^*$ the pure conduction state undergoes a conserved long-wave instability. For $M^*(k_0) = M^* + 4k_0^2$, the pure conduction state is unstable with respect to periodic perturbations with absolute wave number $|\k| \in (0,k_0)$.}
    \label{fig:instability}
\end{figure}

    \subsection{$\Omega$-periodicity and $\Omega$-symmetry}
    Motivated by the linear stability analysis, we are interested in stationary solutions to \Cref{eq:thin-film}, which are periodic and symmetric with respect to a square or hexagonal Fourier lattice. Therefore, we introduce the following notation. 

    A function $h$ is $\square$-periodic with absolute wave number $k_0 > 0$ if
    \begin{equation}\label{eq:periodic}
        h(\x) = h(\x + \bfp_1) = h(\x + \bfp_2)
    \end{equation}
    for the vectors
    \begin{equation*}
        \bfp_1 = \frac{2\pi}{k_0} \bfe_1 \qquad \text{and} \qquad \bfp_2 = \frac{2\pi}{k_0} \bfe_2,
    \end{equation*}
    where $\bfe_1$ and $\bfe_2$ are the unit vectors in $x$- and $y$-direction, respectively.
    Similarly, a function $h$ is $\hexagon$-periodic with absolute wave number $k_0 > 0$ if \eqref{eq:periodic} holds with
    \begin{equation*}
        \bfp_1 = \frac{4\pi}{\sqrt{3}k_0} \bfe_2 \qquad \text{and} \qquad \bfp_2 = \frac{2\pi}{\sqrt{3}k_0}\begin{pmatrix}
			\sqrt{3}\\
			-1
		\end{pmatrix}.
    \end{equation*}
    
    Such functions can be represented as a Fourier series on the Fourier lattice. We will denote by $\bfk_1,\bfk_2\in\R^2$ two distinct vectors with $\abs{\bfk_1}=\abs{\bfk_2}=k_0$ and $\bfk_1\neq\bfk_2$.
    They generate the Fourier lattice
    \begin{equation*}
        \Gamma \coloneqq
        \set{n_1\bfk_1+n_2\bfk_2}{n_1,n_2\in\Z}.
    \end{equation*}
    The Fourier lattice for square patterns is generated by the wave vectors
    \begin{equation*}
        \bfk_1=k_0 \bfe_1
        \qquad\text{and}\qquad
        \bfk_2=k_0 \bfe_2
    \end{equation*}
    and a $\square$-periodic function can be written as
    \begin{equation*}
        \sum_{\boldsymbol{\gamma}\in\Gamma} 
		a_{\boldsymbol{\gamma}}
		\ee^{\ii \boldsymbol{\gamma}\cdot\bfx}
        =
        \sum_{n_1,n_2\in\Zbb} 
		\tilde{a}_{n_1,n_2}
		\ee^{\ii k_0\left(n_1x+n_2y\right)},\quad\bfx\in\R^2.
    \end{equation*}
    The Fourier lattice for hexagonal patterns is generated by
    \begin{equation*}
		\bfk_1=
		k_0 \bfe_1,
        \qquad\text{and}\qquad
		\bfk_2= \frac{k_0}{2}
		\begin{pmatrix}
			-1\\
		    \sqrt{3}
		\end{pmatrix}
	\end{equation*}
    and a $\hexagon$-periodic function can be written as
    \begin{equation*}
        \sum_{\boldsymbol{\gamma}\in\Gamma} 
		a_{\boldsymbol{\gamma}}
		\ee^{\ii \boldsymbol{\gamma}\cdot\bfx}
        =
        \sum_{n_1,n_2\in\Zbb} 
		\tilde{a}_{n_1,n_2}
		\ee^{\ii k_0\left(
			n_1x-\frac{n_2}{2}x+\frac{\sqrt3 n_2}{2}y
			\right)}
        ,\quad\bfx\in\R^2.
    \end{equation*}
    Note that, by abuse of notation, we also denote by $\square$ and $\hexagon$ the corresponding fundamental domains of periodicity, see \Cref{fig:hexcorners}.

    Since \Cref{eq:thin-film} is rotationally invariant, we are only interested in solutions on the square and hexagonal lattice that are invariant under the action of the dihedral groups $D_4$ and $D_6$, respectively.
    These are the symmetries that leave the corresponding Fourier lattices invariant. We will call these functions $\square$-symmetric and $\hexagon$-symmetric, respectively. These symmetry conditions can also be imposed as conditions on the Fourier coefficients. That is, a real-valued function is $\Omega$-symmetric for $\Omega\in \{\square,\hexagon\}$ if and only if the Fourier coefficients satisfy
    \begin{equation}\label{eq:cond-symmetry}
        a_{\boldsymbol{\gamma}}=\overline{a_{-\boldsymbol{\gamma}}},\quad\text{and}\quad a_{\boldsymbol{\gamma}}=a_{S\boldsymbol{\gamma}}\quad\text{for all}\quad S\in
    \begin{cases*}
	   D_4 & if $\Omega=\square$,\\
	   D_6 & if $\Omega=\hexagon$.
	\end{cases*}
    \end{equation}
    Note that in both cases the coefficients are necessarily real. Furthermore, the Fourier series can be represented in cosine form, which is also a consequence of symmetry.

    For later convenience, we also introduce the sets of zeros of $\hat{\Lcal}_{M^\ast(k_0)}$ which lie on the respective Fourier lattice
    \begin{equation}\label{eq:def-Lambda}
        \Lambda_\square\coloneqq
        \left\{\pm\bfk_1,\pm\bfk_2\right\}
        \qquad\text{and}\qquad
        \Lambda_\hexagon\coloneqq
        \left\{\pm\bfk_1,\pm\bfk_2,\pm(\bfk_1+\bfk_2)\right\}.
    \end{equation}
    These are exactly the Fourier modes which destabilise at $M = M^*(k_0)$. 

    Lastly, there is a consequential qualitative difference between the square and the hexagonal Fourier lattice:
    in the hexagonal setup, we can combine $\bfk_1$ and $\bfk_2$ nontrivially to find a third Fourier vector of the same length
    \begin{equation*}
        \bfk_3=-\sprlr{\bfk_1+\bfk_2}\quad\text{and}\quad\abslr{\bfk_3}=k_0.
    \end{equation*}
    This makes the hexagonal Fourier lattice resonant, and in the literature $\bfk_1,\bfk_2,\bfk_3$ is referred to as a resonant triad \cite{uecker2021}.

     \begin{figure}[h!]
	\centering
    \includegraphics[width=0.4\textwidth]{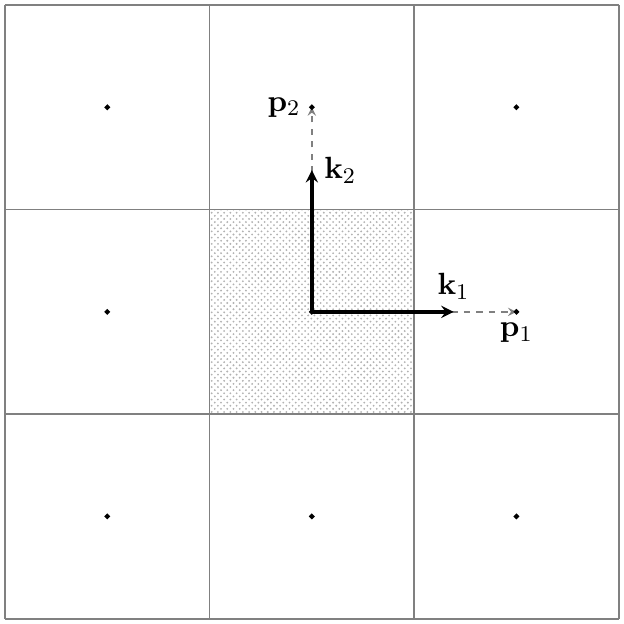}\hspace{2cm}
    \includegraphics[width=0.4\textwidth]{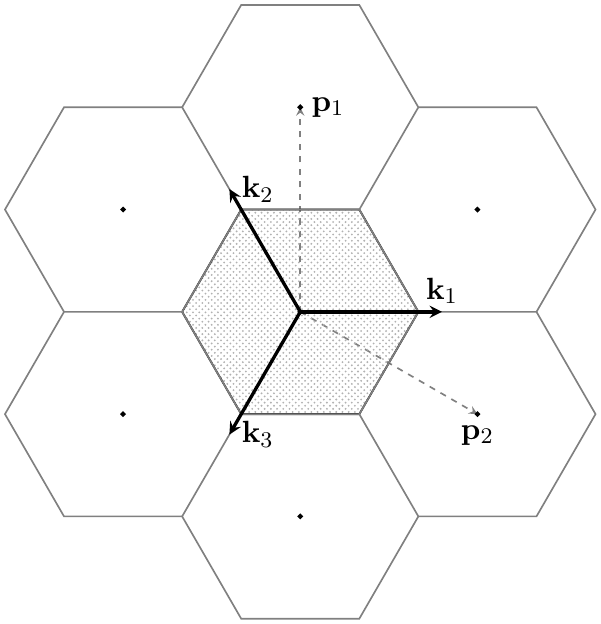}
	\caption{Lattices for square (left) and hexagonal (right) patterns. Black arrows indicate the generators of the Fourier lattices; grey lines show the directions of pattern periodicity. The grey dotted regions represent the fundamental domains of periodicity.}
	\label{fig:hexcorners}
    \end{figure}

    \subsection{Stationary solutions}

    In the remainder of the paper, we will be concerned with $\Omega$-symmetric and -periodic stationary solutions bifurcating from the pure conduction state $\Bar{h}=1$. To this end, we notice that we can rewrite the stationary thin-film equation 
	\begin{equation}
		\label{eq:stationary-thin-film}
		\nabla\cdot
		\sprlr{
			h^3
			\sprlr{\nabla\Delta h-g\nabla h}
			+
			M\frac{h^2}{(1+h)^2}\nabla h
		}
		=0.
	\end{equation}
	as a second-order elliptic problem with a carefully chosen nonlocal right-hand side. We write $h=1+v$, where $v>-1$, yielding
	\begin{equation*}
		\nabla\cdot
		\sprlr{
			(1+v)^3
			\sprlr{\nabla\Delta v-g\nabla v
				+
				M\frac{1}{(1+v)(2+v)^2}\nabla v
		}}
		=0.
	\end{equation*}
	Using that
	\begin{equation*}
		\frac{1}{(1+v)(2+v)^2}\nabla v
		=
		\nabla \sprlr{\frac{1}{2+v}+\log\sprlr{\frac{1+v}{2+v}}}
	\end{equation*}
	we find that 
	\begin{equation*}
		\nabla\cdot
		\sprlr{
			(1+v)^3
			\nabla\sprlr{\Delta v-gv
				+
				M
				\sprlr{\frac{1}{2+v}+\log\sprlr{\frac{1+v}{2+v}}
		}}}
		=0.
	\end{equation*}
	This problem is of the form $\nabla\cdot\sprlr{a(v)\nabla G(v)}=0$. 
	If $v>-1$, the function $G(v)$ is well-defined and $\Omega$-periodic, and $a(v)>0$ uniformly in $\Omega$.
	By uniqueness for this elliptic Neumann problem
	\begin{equation*}
		\nabla\cdot\sprlr{a(v)\nabla G(v)}=0
		\quad\iff\quad
		G(v)\equiv \tilde{K},\;\tilde{K}\in\Rbb.
	\end{equation*}
	Hence, we can reformulate the stationary thin-film equation \eqref{eq:stationary-thin-film} as
	\begin{equation}
		\label{eq:integrated-stationary-problem}
		\Delta v-gv
		+
		M
		\sprlr{\frac{1}{2+v}+\log\sprlr{\frac{1+v}{2+v}}}
		=MK(v),
	\end{equation}
	where $MK(v)$ is the integration constant depending on $v$ and $M>0$.
	We choose
	\begin{equation}
		\label{eq:K}
		K(v)=
		\fint_\Omega\frac{1}{2+v}+\log\sprlr{\frac{1+v}{2+v}}\de \bfx,
	\end{equation}
	which ensures that \Cref{eq:integrated-stationary-problem} has nontrivial solutions with vanishing mean.

	\section{Local bifurcation theory}
    \label{sec loc bif}
    In this section, we establish the bifurcation of periodic patterns from the pure conduction state.
    The latter destabilises at $M=M^\ast$ with respect to long-wave perturbations, and we choose $M$ as bifurcation parameter.
    Having set the size of the domain $\Omega\in\{\square,\hexagon\}$ by fixing $k_0$, we now determine the bifurcation point $M^*(k_0)$ at which patterns with the same size bifurcate.
    Consider 
	\begin{equation*}
		M^\ast(k_0)=M^\ast+4 k_0^2=4g+4 k_0^2
	\end{equation*}
	such that for $M=M^\ast(k_0)$ the dispersion relation \eqref{disp rel} vanishes at $\abs{\bfk}=k_0$
	\begin{equation*}
		\lambda_{M=M^\ast(k_0)}(\bfk)=0, \quad\text{if}\;\abs{\bfk}=k_0.
	\end{equation*}
	Further, it holds $\partial_M \lambda_{M}(\k)|_{M=M^*(k_0)} > 0$ for $|\k| = k_0$. Hence, for the choice $M = M^*(k_0)$, we expect the bifurcation of an $\Omega$-periodic solution from the pure conduction state.

    \begin{remark}
    \label{rema other bifurcation points}
        We note that at $(v,M)=\sprlr{0,M^\ast(nk_0)}$, where $n\in\N_+$, we expect the bifurcation of periodic solutions such that $\Omega$ fits $n$ periods per dimension.
    \end{remark}
    
	We now set up the corresponding bifurcation problem. First, we introduce the following function spaces
	\begin{align*}
		\Xcal&\coloneqq
		\set{v\in\Hrm_\per^2(\Omega)}
		{v\;\textrm{is $\Omega$-symmetric,}\;
			\fint_\Omega v \de x=0},\\
		\Ycal&\coloneqq
		\set{v\in\Lrm_\per^2(\Omega)}
		{v\;\textrm{is $\Omega$-symmetric,}\;
			\fint_\Omega v \de x=0},\\
		\Ucal&\coloneqq
		\set{v\in\Xcal}
		{v>-1},
	\end{align*}
	where $\Hrm^k_\per(\Omega)$ is the Sobolev space of $\Omega$-periodic functions and $\Xcal$ and $\Ycal$ carry the $\Hrm_\per^2(\Omega)$- and the $\Lrm_\per^2(\Omega)$-norm, respectively.
	With this notation, we can define the analytic bifurcation function
	\begin{align*}
		F\colon\Ucal\times\Rbb\longrightarrow\Ycal \; \colon \; F(v,M)=
		\Delta v-gv
		+
		M\sprlr{\frac{1}{2+v}+\log\sprlr{\frac{1+v}{2+v}}}
		-MK(v),
	\end{align*}
	where the constant $K(v)$ is defined in \eqref{eq:K}.
	Note that the Laplacian is invariant under the $\Omega$-symmetries and our choice of $K(v)$ ensures that $F(v,M)$ has vanishing mean for any $(v,M) \in \Ucal\times\R$.
	The bifurcation problem therefore is 
	\begin{equation}
		\label{bif problem}
		F(v,M)=0.
	\end{equation}

    In the following, we establish local bifurcation results for the square and the hexagonal lattice, respectively. 
    We treat both cases separately, since the resonance in the hexagonal lattice makes the quadratic interactions of critical Fourier modes relevant.
    As a consequence, the bifurcation of hexagonal patterns is generally transcritical, while square patterns arise from a pitchfork bifurcation.

    \subsection{Local bifurcation of square patterns}    
	Assuming $\square$-symmetry, it turns out that a bifurcation from a one-dimensional kernel occurs and the Crandall--Rabinowitz local bifurcation theorem in \cite[Thm. 8.3.1]{buffoni2003} applies.
    We obtain the following result.
	\begin{theorem}[Local bifurcation on the square lattice]
		\label{thm loc bif squ}
		Fix $k_0>0$ and let $\Omega=\square$. Then a subcritical pitchfork bifurcation occurs at $(v(0),M(0)) = (0,M^*(k_0))=(0,M^\ast+4 k_0^2)$, that is, there exists $\varepsilon>0$ and a branch of solutions 
		\begin{equation*}
			\Ccal_\loc\coloneqq\set{\sprlr{v(s),M(s)}}{s\in(-\varepsilon,\varepsilon)}
			\subset \Ucal\times\Rbb
		\end{equation*}
		to the bifurcation problem \eqref{bif problem}
		with expansions
		\begin{equation}
			\label{expansions squ}
			\begin{split}
				v(s) & = s\hat{\xi}_0+\tau(s),
				\\
				M(s) & = M^\ast(k_0)-\frac{\sprlr{g+k_0^2}\sprlr{104+203k_0^2}}
            {24\pi^2\sprlr{1+k_0^2+k_0^4}}
				s^2+\Ocal(\abs{s}^3),
			\end{split}
		\end{equation}
		where $\tau(s)=\Ocal(\abs{s}^2)$ in $\Xcal$ as $s\to0$, and $\hat{\xi}_0$ is the normalised kernel element defined in \eqref{kernel element squ}.
	\end{theorem}
    
	\begin{proof}
		The proof follows from a straightforward application of the Crandall-Rabinowitz theorem, see e.g.~\cite[Thm.~8.3.1]{buffoni2003}. Note that by our choice of $K(v)$ in \eqref{eq:K} we obtain $F(0,M)=0$ for all $M\in\Rbb$.
		
		\noindent\textbf{Step 1:} Existence of local bifurcation curve. We need to check the following three conditions:
		the linearisation $L$ of $F\sprlr{\cdot,M^*(k_0)}$ about $v\equiv0$ is Fredholm of index zero and has a one-dimensional kernel. Finally, the transversality condition \eqref{transversality} needs to be satisfied. 
		
		We calculate the linearisation of $F\sprlr{\cdot,M^\ast(k_0)}$ about $v\equiv 0$
		\begin{equation*}
			L\coloneqq \partial_v F\sprlr{0,M^\ast(k_0)}
			=\Delta+\sprlr{\frac{1}{4}M^\ast(k_0)-g}
			=\Delta+k_0^2.
		\end{equation*}
		Note that for $w \in \Xcal$ by the vanishing mean condition the derivative of the constant term vanishes
		\begin{equation*}
			\partial_v K(0)(w)
			=
			\fint_\Omega \frac{w}{4} \de x
			=0.
		\end{equation*}
		
		The Fourier symbol of $L$ is given by
		\begin{equation*}
			\hat{L}(\bfk) = k_0^2 \sprlr{1-\abslr{\bfk}^2},\qquad \bfk=
			\begin{pmatrix}
				k_1\\
				k_2
			\end{pmatrix}
			\in\Zbb^2,
		\end{equation*}
		and therefore
		\begin{equation*}
			\ker \hat{L}=\set{\bfk\in\Zbb^2}{\abs{\bfk}=1}.
		\end{equation*}
		By $\square$-symmetry, this reduces to the one-dimensional subspace of $\Xcal$
		\begin{equation}
			\label{ker L}
			\ker L = \mathrm{span}
			\dprlr{
            \sum_{\boldsymbol{\gamma}\in\Lambda_\square}
            a_{\boldsymbol{\gamma}}\ee^{\ii \boldsymbol{\gamma}\cdot\bfx}
            \ \ \textup{\textbf{:}}\ \  
            \textrm{$a_{\boldsymbol{\gamma}}$ satisfy \eqref{eq:cond-symmetry}}
            }
            = \mathrm{span}
			\dprlr{\hat{\xi}_0}
		\end{equation}
        and we denote a kernel element and its normalisation by
        \begin{equation}
        \label{kernel element squ}
            \xi_0\coloneqq\ee^{\ii k_0x}+\ee^{-\ii k_0x}+\ee^{\ii k_0y}+\ee^{-\ii k_0y},
            \qquad
            \hat{\xi}_0\coloneqq\frac{\xi_0}{\normlr{\xi_0}_\Xcal}.
        \end{equation}
        
		The range of $L$ in $\Ycal$ has codimension one, since
		\begin{equation*}
			\ran L 
			=\mathrm{span}
			\dprlr{
            \sum_{\boldsymbol{\gamma}\in\Gamma \setminus\Lambda_\square}
            a_{\boldsymbol{\gamma}}\ee^{\ii \boldsymbol{\gamma}\cdot\bfx}
            \ \ \textup{\textbf{:}}\ \  
            \textrm{$a_{\boldsymbol{\gamma}}$ satisfy \eqref{eq:cond-symmetry}}}, 
		\end{equation*}
		which proves that $L$ is Fredholm of index zero with one-dimensional kernel.
		
		Finally, we check the transversality condition
		\begin{equation}
			\label{transversality}
			\partial^2_{v,M}
			F\sprlr{0,M^\ast(k_0)}(\xi_0,1)
			\notin \ran L.
		\end{equation}
		Therefore, we calculate
		\begin{equation*}
			\partial^2_{v,M}F\sprlr{0,M^\ast(k_0)}(\xi_0,1)
			=
			\partial_v\sprlr{
				\frac{1}{2+v}+\log\sprlr{\frac{1+v}{2+v}}-K(v)
			}\bigg\rvert_{v=0}(\xi_0,1)
			=
			\frac{1}{4}\xi_0
			\notin \ran L.
		\end{equation*}
        Hence, the Crandall–Rabinowitz theorem applies and gives the existence of a local bifurcation curve $\Ccal_\loc$.
		
		\noindent\textbf{Step 2:} Expansion of solutions and local behavior of bifurcation curve. Using \cite[Cor. I.5.2]{kielhofer2012}, we obtain
		\begin{equation*}
			v(s)=s\hat{\xi}_0+\tau(s),
		\end{equation*}
		where $\tau(s)=\Ocal(\abs{s}^2)$ in $\Xcal$ as $s\to0$.
		Next, we calculate the Taylor expansion of $M(s)$ about $s=0$.
		The first derivative of $M(s)$ at $s=0$ vanishes by \cite[Eq. I.6.3]{kielhofer2012} and 
		\begin{equation*}
			\partial_{vv}^2F\sprlr{0,M^\ast(k_0)}\left[\hat{\xi}_0,\hat{\xi}_0\right]=-\frac{M^\ast(k_0)}{2}\hat{\xi}_0^2+\frac{M^\ast(k_0)}{2}\fint_\Omega \hat{\xi}_0^2 \de x\in\ran L.
		\end{equation*}
		The second derivative is given by \cite[Eq. I.6.11]{kielhofer2012}
		\begin{equation*}
			\ddot{M}(0)=-\frac13\frac{
				\int_\Omega \sprlr{Q\partial_{vvv}^3F\sprlr{0,M^\ast(k_0)}\left[\hat{\xi}_0,\hat{\xi}_0,\hat{\xi}_0\right]-3Q\partial_{vv}^2F\sprlr{0,M^\ast(k_0)}\left[\hat{\xi}_0,A\right]}
				\cdot\hat{\xi}_0\de x
			}{
				Q\int_\Omega
				\partial_{v,M}^2F\sprlr{0,M^\ast(k_0)}\hat{\xi}_0\cdot\hat{\xi}_0
				\de x
			},
		\end{equation*}
		where $Q$ is the projection in $\Ycal$ along $\ran L$, and 
        \begin{align*}
			A&\coloneqq L^{-1}(\Id-Q)\partial_{vv}^2F\sprlr{0,M^\ast(k_0)}\left[\hat{\xi}_0,\hat{\xi}_0\right].
        \end{align*}
		This yields
		\begin{equation*}
			\ddot{M}(0)=-\frac{\sprlr{g+k_0^2}\sprlr{104+203k_0^2}}
            {24\pi^2\sprlr{1+k_0^2+k_0^4}},
		\end{equation*}
        which completes the proof of the expansion \eqref{expansions squ}.
	\end{proof}
	
	\subsection{Local bifurcation of hexagonal patterns}
	\label{subsec hex setup bif}
    In this subsection we prove bifurcation of hexagonal patterns. The proof is similar to the square case and we only provide the necessary modifications. The main difference comes from the fact that hexagonal patterns originate from a transcritical bifurcation.
    
    \begin{theorem}[Local bifurcation on the hexagonal lattice]
		\label{thm loc bif hex}
		Fix $k_0>0$ and let $\Omega=\hexagon$. Then a transcritical bifurcation occurs at $(v(0),M(0)) = (0,M^*(k_0))=(0,M^\ast+4 k_0^2)$, that is, there exists $\varepsilon>0$ and a branch of solutions 
		\begin{equation*}
			\Ccal_\loc\coloneqq\set{\sprlr{v(s),M(s)}}{s\in(-\varepsilon,\varepsilon)}
			\subset \Ucal\times\Rbb
		\end{equation*}
		to the bifurcation problem \eqref{bif problem}
		with expansions
		\begin{align}
			\label{expansions hex}
			\begin{split}
				v(s)&=s\hat{\psi}_0+\tau(s),
				\\
				M(s)&=M^\ast(k_0)+
                \frac{2\sprlr{g+k_0^2}}{3^\frac14\pi \sqrt{k_0^2+1+\frac{1}{k_0^2}}}
                s+\Ocal(\abs{s}^2),
			\end{split}
		\end{align}
		where $\tau(s)=\Ocal(\abs{s}^2)$ in $\Xcal$ as $s\to0$, and $\hat{\psi}_0$ is the normalised kernel element defined in \eqref{kernel element hex}.
	\end{theorem} 

\begin{proof}
    Fortunately, $L=\Delta+k_0^2$ still acts diagonally on Fourier modes, and we can compute
	\begin{align}
		\label{L on hexagonal mode}
		L
		\ee^{\ii k_0\left(n_1x-\frac{n_2}{2}x+\frac{\sqrt3 n_2}{2}y\right)}
		=
		k_0^2\sprlr{1-n_1^2-n_2^2+n_1n_2}
		\ee^{\ii k_0\left(n_1x-\frac{n_2}{2}x+\frac{\sqrt3 n_2}{2}y\right)}.
	\end{align}
	A Fourier mode with wave number $\bfk = n_1 \k_1 + n_2 \k_2$ lies in $\ker L$ if and only if $\bfk \in \Lambda_{\hexagon}$.
    Hence, by $\hexagon$-symmetry, $\ker L$ is the one-dimensional subspace of $\Xcal$ given by
	\begin{equation*}
		\ker L = \mathrm{span}
		\dprlr{
            \sum_{\boldsymbol{\gamma}\in\Lambda_\hexagon^\ast}
            a_{\boldsymbol{\gamma}}\ee^{\ii \boldsymbol{\gamma}\cdot\bfx}
            \ \ \textup{\textbf{:}}\ \  
            \textrm{$a_{\boldsymbol{\gamma}}$ satisfy \eqref{eq:cond-symmetry}}
        }
        = \mathrm{span}
		\dprlr{\hat{\psi}_0}
	\end{equation*}
    and we denote a kernel element and its normalisation by
    \begin{equation}
    \label{kernel element hex}
        \psi_0\coloneqq
        \ee^{\ii k_0x}
        +\ee^{\ii k_0\left(-\frac{1}{2}x+\frac{\sqrt3}{2}y\right)}
        +\ee^{\ii k_0\left(x-\frac{1}{2}x+\frac{\sqrt3}{2}y\right)}
        +\,\text{c.c.},
        \qquad
        \hat{\psi}_0\coloneqq\frac{\psi_0}{\normlr{\psi_0}_\Xcal},
    \end{equation}
    where $\text{c.c.}$ denotes the complex conjugated terms.
    Therefore, similar to the square case the existence of a local bifurcation curve follows by Crandall--Rabinowitz. 
    
	The main difference to the square case is the behavior of $M(s)$ close to $s=0$.
	In particular, 
	\begin{equation*}
		\dot{M}(0)=-\frac12\frac{\langle\partial_{vv}^2F\sprlr{0,M^\ast(k_0)}\left[\hat{\psi}_0,\hat{\psi}_0\right],\hat{\psi}_0\rangle}{\langle\partial_{v,M}^2F\sprlr{0,M^\ast(k_0)}\left[\hat{\psi}_0,1\right],\hat{\psi}_0\rangle} = \frac{8\sprlr{g+k_0^2}}{\normlr{\psi_0}_\Xcal}=\frac{2\sprlr{g+k_0^2}}{3^\frac14\pi \sqrt{k_0^2+1+\frac{1}{k_0^2}}} >0
	\end{equation*}
    by \cite[Eq. I.6.3]{kielhofer2012}.
	Therefore, a transcritical bifurcation takes place and the proof is complete.
\end{proof}

    \subsection{Spectral stability analysis of nontrivial solutions}\label{sec:spec-stability}

    We now briefly discuss the stability of the solutions on the local bifurcation branch as solutions to the thin-film equation \eqref{eq:thin-film} with respect to different types of perturbations using spectral analysis. First, we consider the case of $\Lrm^2$-localised perturbations. It is straightforward to see that the periodic solutions on the local bifurcation branch are spectrally unstable with respect to this class of perturbations. Indeed, we can write the linearisation about the periodic solutions as $\Lcal_{\mathrm{per},M(s)} = \Lcal_{M^*(k_0)} + s \Lcal_1(s)$ with $s \in (-\varepsilon, \varepsilon)$, where $\Lcal_{M^*(k_0)}$ is the linearisation about the pure conduction state $\Bar{h}\equiv1$, see \eqref{linearised operator}. Since $\Lcal_{M^*(k_0)}$ has $\Lrm^2$-spectrum in the positive complex half-plane, we can apply the perturbation results in \cite[Thm.~IV.3.1]{kato1995} using that $\Lcal_1(s)$ is $\Lcal_{\mathrm{per},M(s)}$-bounded, see also \cite[Sec. 6.1]{bruell2024}. This yields that $\Lcal_{\mathrm{per},M(s)}$ has $\Lrm^2$-spectrum in the positive complex half-plane. Therefore, the solutions are spectrally unstable.

    Next, we consider perturbations, which are periodic with respect to a Fourier lattice $\tilde{\Gamma}$ and satisfy the symmetry condition \eqref{eq:cond-symmetry}. Then, we distinguish three cases. If the perturbations are superharmonic, that is, $\tilde{\Gamma} = N\Gamma$ for $N \in \N$, $N > 1$, the periodic solutions on the local bifurcation branch are spectrally stable. This follows from a similar perturbation argument as in the case of $\Lrm^2$-perturbations. If the perturbations are subharmonic, that is, $\tilde{\Gamma} = \tfrac{1}{N} \Gamma$ for $N \in \N$, $N > 1$, the periodic solutions on the local bifurcation branch are spectrally unstable again following a similar perturbation argument.

    \subsection{Nonlinear stability analysis for co-periodic perturbations} \label{sec:stability-coperiodic}

    Finally, we consider the case where $\tilde{\Gamma} = \Gamma$, that is, the perturbations are co-periodic. In this case, we can analyse the stability of solutions on the local bifurcation branch using a centre manifold reduction. For this, we consider solutions to the thin-film equation \eqref{eq:thin-film} which are periodic with respect to $\Gamma$ and satisfy the symmetry conditions \eqref{eq:cond-symmetry}. Then, at $M = M^*(k_0)$, the linearisation has a simple eigenvalue $\lambda = 0$ on the imaginary axis and there is a spectral gap around the imaginary axis due to the discreteness of the Fourier lattice. Therefore, there exists a one-dimensional invariant center manifold, see, e.g.~\cite{haragus2011b}. 
    
    In the case $\Omega = \square$, we recall that square patterns arise from a subcritical pitchfork bifurcation. Therefore, the normal form, see, e.g.~\cite{haragus2011b}, of the reduced equation on the centre manifold reads as
    \begin{equation}\label{eq:red-eq-square}
        \partial_T A = \alpha_1 M_0 A + \alpha_2 A^3 + \rho_\mathrm{sq}(A)
    \end{equation}
    with $\alpha_1, \alpha_2 > 0$ and $\rho_\mathrm{sq}(A) = \Ocal(\mu|A|^4)$, where $\mu^2 M_0 = M(s) - M^*(k_0)$ and $0 < \mu \ll 1$ denotes the distance from the bifurcation point, $T = \mu^2 t$, and $\mu A$ is the amplitude of the kernel element $\hat{\xi}_0$. In particular, square patterns in \eqref{eq:thin-film} correspond to a non-trivial fixed point in \eqref{eq:red-eq-square}. Since all small, bounded solutions to \eqref{eq:thin-film} which are periodic with respect to $\Gamma$ and satisfy the symmetry condition \eqref{eq:cond-symmetry} lie on the centre manifold, the stability of square patterns in \eqref{eq:thin-film} with respect to co-periodic and $\square$-symmetric perturbations corresponds to the stability of the nontrivial fixed point in \eqref{eq:red-eq-square}. A straightforward calculation shows that the linearisation about the fixed point $A_{\mathrm{sq}} = \pm \sqrt{-\tfrac{\alpha_1 M_0}{\alpha_2}}$ is given by $\Lcal_\mathrm{sq} = -2M_0\alpha_1 > 0$ since $M_0 < 0$ and $\alpha_1 > 0$. Therefore, square patterns are linearly unstable. For a more detailed analysis in a similar thin-film system, we refer to \cite{hilder2024a}.

    In the case $\Omega = \hexagon$, hexagonal patterns arise from a transcritical bifurcation, see \Cref{thm loc bif hex}. Thus, the normal form in this case reads as
    \begin{equation}\label{eq:red-eq-hex}
        \partial_T B = \alpha_1 M_0 B + \beta B^2 + \rho_\mathrm{hex}(B),
    \end{equation}
    where $\alpha_1 > 0$ is the same coefficient as in \eqref{eq:red-eq-square} and $\rho_\mathrm{hex}(B) = \Ocal(\mu^2 |B|^3)$, $T = \mu^2 t$ and $\mu^2 B$ is the amplitude of the kernel element $\hat{\psi}_0$. Since fixed points of \eqref{eq:red-eq-hex} have to agree with the bifurcation diagram obtained in \Cref{thm loc bif hex}, we find that $\beta = - \alpha_1 \dot{M}(0) < 0$. Again, linearising about the non-trivial fixed point $B_\mathrm{hex} = \tfrac{-\alpha_1 M_0}{\beta}$, we find that $\Lcal_\mathrm{hex} = -\alpha_1 M_0$, which is positive for $M_0 < 0$ and negative for $M_0 > 0$. Therefore, we find that down-hexagons are linearly unstable whereas up-hexagons are linearly stable with respect to co-periodic and $\hexagon$-symmetric perturbations.

    \begin{remark}
        We point out that in contrast to the localised or non-co-periodic case, we obtain nonlinear (in)stability in the co-periodic case. This follows from the observation that the dynamics are described by a one-dimensional ODE with a sufficiently regular nonlinearity, and therefore, we obtain the nonlinear behaviour from a Duhamel argument.
    \end{remark}

    \begin{remark}
        Comparing the stability results for hexagonal patterns presented here with the stability discussion in \cite{hoyle2007}, we obtain a transition from unstable to stable by passing through the bifurcation point, while \cite{hoyle2007} finds that solutions on both local branches are unstable. This is the case because we restrict to the class of $\hexagon$-symmetric perturbations, while the more general class of perturbations which are periodic with respect to $\Gamma$, but do not necessarily satisfy the symmetry conditions \eqref{eq:cond-symmetry}, is considered in \cite{hoyle2007}.
    \end{remark}

    \begin{remark}
        In the one-dimensional case \cite{bruell2024}, one can use the results in \cite{laugesen2000,laugesen2002} to obtain linear and energy instability even for large-amplitude stationary periodic solutions with respect to co-periodic perturbations. However, these arguments do not seem to generalise to the two-dimensional case, see \Cref{sec:discussion} for a more detailed discussion.
    \end{remark}
	
	\section{Global bifurcation theory}
    \label{sec glob bif}
    In this section, we extend the local bifurcation curves to global ones using the analytic global bifurcation theorem \cite[Thm. 6]{constantin2016a}.
    Thereafter, we perform a detailed qualitative analysis allowing us to exclude the 'closed loop'-alternative \ref{loop} in \cref{thm glob bif}, see \cref{prop not loop}.
    We achieve this using the result \cite[Thm. 9.2.2]{buffoni2003} on global bifurcation in cones, which also establishes a nodal property.
    \Cref{prop no derivative blowup} excludes a mere blow-up of derivatives along the bifurcation branch by using a uniform regularity estimate.
    Lastly, we present a conditional result on film rupture in \cref{prop cond rupture}: 
    if the bifurcation parameter remains bounded, i.e. $\abs{M(s)}$ admits a bound independent of $s$, then the 'boundary'-alternative \ref{approach boundary} in \cref{thm glob bif} has to occur. Specifically, there exists a subsequence $(v(s_j))_{j\in\N}$ so that $\min v(s_j) \to -1$ as $j\to \infty$. However, unlike in the one-dimensional case, we cannot prove sufficiently strong uniform bounds to be able to prove the existence of a stationary weak solution exhibiting film rupture.
    
	\subsection{Existence of a global bifurcation branch}
	We first state the global bifurcation result. It is independent of $\Omega\in\{\square,\hexagon\}$, meaning that both local bifurcation curves can be extended, and the proof is identical.
    
	\begin{theorem}
		\label{thm glob bif}
		Let $k_0>0$, $\Omega\in\{\square,\hexagon\}$ and
		\begin{equation*}
			\Ccal_\loc=\set{\sprlr{v(s),M(s)}}{s\in(-\varepsilon,\varepsilon)}
			\subset \Ucal\times\Rbb
		\end{equation*}
		be the bifurcation branch with $(v(0),M(0)) = (0,M^*(k_0))$ obtained in \cref{thm loc bif squ} or \cref{thm loc bif hex}, respectively.
		There exists a globally defined continuous curve 
		\begin{equation*}
			\Ccal\coloneqq\set{\sprlr{v(s),M(s)}}{s\in\Rbb}
			\subset \Ucal\times\Rbb
		\end{equation*}
		consisting of smooth $\Omega$-periodic and -symmetric solutions to the bifurcation problem \eqref{bif problem}.
		This curve extends the local bifurcation branch such that at least one of the following occurs:
		\begin{enumerate}[label=(\roman*)]
			\item\label{blow up} 
			$\normlr{\sprlr{v(s),M(s)}}_{\Xcal\times\Rbb}\longrightarrow\infty$ as $s\to\pm\infty$;
			\item\label{approach boundary} 
			there exists a subsequence $s_j\rightarrow\infty$ such that $\dist\sprlr{\bigl(v(s_j),M(s_j)\bigr),\partial\Ucal\times\R}\longrightarrow0$ as $j\to\pm\infty$;
			\item\label{loop}
			$\sprlr{v(s),M(s)}$ is a closed loop.
		\end{enumerate}
	\end{theorem}
	
	\begin{remark}
		 In \cite[Thm. 6]{constantin2016a}, the alternatives \ref{blow up} and \ref{approach boundary} are replaced by the stronger alternative
		\begin{enumerate}[label=(\roman*)']
			\item\label{strong alternative}
			For all bounded and closed sets $Q\subset\Ucal\times\Rbb$ there exists $s_0\in\Rbb$ such that $\sprlr{v(s),M(s)}\notin Q$ for all $\abs{s}>s_0$,
		\end{enumerate}
		see also \cite[Rem. 8]{constantin2016a}. However, \ref{strong alternative} implies that \ref{blow up} or \ref{approach boundary} occurs: we can prove this by defining 
		\begin{equation}
        \label{Qj sets}
			Q_j\coloneqq U_j\times[-j,j],
			\qquad\text{where}\qquad 
			U_j\coloneqq\set{v\in\Xcal}{-1 + \frac{1}{j}\leq v\leq j},
			\qquad\text{for}\,
			j\in\Nbb_+.
		\end{equation}
		
		Alternative \ref{strong alternative} is indeed stronger: while \ref{blow up} implies \ref{strong alternative}, the alternative \ref{approach boundary} does not. 
		One can easily think of a curve which satisfies \ref{approach boundary}, while it does not satisfy any of \ref{loop} and \ref{strong alternative}.
	\end{remark}
	
	\begin{remark}
		As pointed out in \cite[Rem. 8]{constantin2016a}, the alternatives are not exhaustive if we replace \ref{approach boundary} with
		\begin{enumerate}[label=(\roman*)', start=2]
			\item\label{wrong alternative}
			$\dist\sprlr{\sprlr{v(s),M(s)},\partial\Ucal\times\R}\longrightarrow0$ as $s\to\infty$.
		\end{enumerate}
		As a consequence, we need to pass to a subsequence in the conditional film rupture result in \cref{prop cond rupture}.
	\end{remark}
	
	\begin{proof}[Proof of \Cref{thm glob bif}]
		We need to check the conditions of \cite[Thm. 6]{constantin2016a}.
		Condition $(H_1)$ is satisfied as $(0,M)\in\Ucal\times\Rbb$ and $F(0,M)=0$ for all $M\in\Rbb$.
		We have already established that the kernel is one-dimensional as well as the transversality condition in the proofs of \cref{thm loc bif squ,thm loc bif hex}, which yields validity of $(H_2)$.
		It remains to check two conditions:
		\begin{enumerate}[label=($H_{\arabic*}$),start=3]
			\item\label{prop global fredholm}
			$\partial_v F(v,M)$ is a Fredholm operator of index zero, whenever $(v,M)\in\Scal$, where \\$\Scal\coloneqq\set{(v,M)\in\Ucal\times\Rbb}{F(v,M)=0}$ is the set of solutions to the bifurcation problem \eqref{bif problem};
			\item\label{prop global compact}
			for some sequence $\sprlr{Q_j}_{j\in\Nbb}$ of bounded closed sets $Q_j\subset\Ucal\times\Rbb$ with $\bigcup_{j\in\Nbb}Q_j=\Ucal\times\Rbb$, the set $\Scal\cap Q_j$ is compact for all $j\in\Nbb$.
		\end{enumerate}
		We compute for $v_0\in\Xcal$
		\begin{equation*}
			\partial_v F(v,M)v_0
			=
			\Delta v_0 + \sprlr{\frac{M}{(1+v)(2+v)^2}-g} v_0 - \fint_\Omega\frac{M}{(1+v)(2+v)^2}v_0\de x,
		\end{equation*}
		which is a bounded nonlocal operator.
		By the vanishing mean condition and by $\Omega$-periodicity, the Laplacian $\Delta\colon\Xcal\rightarrow\Ycal$ is invertible with bounded inverse, and we may write
		\begin{equation*}
			\partial_v F(v,M)=\Delta\circ\Delta^{-1}\partial_v F(v,M).
		\end{equation*}
		It is useful to define
		\begin{equation*}
			\Tilde{L}_{v,M} v_0 
			\coloneqq \Delta^{-1}\partial_v F(v,M) v_0
			= v_0 + \Delta^{-1}
			\sprlr{
				\sprlr{\frac{M}{(1+v)(2+v)^2}-g} v_0 - \fint_\Omega\frac{M}{(1+v)(2+v)^2}v_0\de x
			}.
		\end{equation*}
		As $\Delta^{-1}\colon\Xcal\to\Xcal$ is compact, we find that $\Tilde{L}_{v,M}$ is of the form identity plus compact and thus Fredholm of index zero. 
		The Laplacian $\Delta\colon\Xcal\rightarrow\Ycal$ is also Fredholm of index zero and so is their composition $\partial_v F(v,M)=\Delta\circ\Tilde{L}_{v,M}$.
		
		We have proved \ref{prop global fredholm}.
		Regarding \ref{prop global compact}, choose the sets $Q_j$ as in \eqref{Qj sets}.
        By standard elliptic regularity theory, solutions to the bifurcation problem \eqref{bif problem} are smooth, and the compact embedding of $\Xcal$ into $\Ycal$ yields the claim.
	\end{proof}
	
	\subsection{Behaviour of the global bifurcation branch}
	We proceed to rule out the closed-loop alternative \ref{loop} in \cref{thm glob bif}.
	Thereafter, we provide a qualitative description of alternative \ref{blow up}: we prove that if blow-up occurs, then $(v(s),M(s))_{s\in\R}$ is already unbounded in $\Ycal\times \R$.
    First, we prove a useful estimate for the local expansions $v(s) = s\hat{\xi}_0 + \tau(s)$ on the square lattice, see \Cref{expansions squ}, and $v(s) = s \hat{\psi}_0 + \tau(s)$ on the hexagonal lattice, see \Cref{expansions hex}.
	\begin{lemma}
		\label{lemma higher tau}
		The remainder term satisfies $\tau(s)=\Ocal(\abs{s}^2)$, as $s\to 0$, in $\Hrm_\per^m(\Omega)$ for all $m\geq 0$.
	\end{lemma}
	\begin{proof}
		Let $\Omega=\square$.
		Recall that $\hat{\xi}_0$ is defined in \Cref{kernel element squ}.
		By the expansion \eqref{expansions squ} of the local bifurcation branch, we know $v(s)= s\hat{\xi}_0+\tau(s)$, where $\tau=\Ocal(\abs{s}^2)$ in $\Xcal$.
		We write $f_s(v)\coloneqq F(v,M(s))-\Delta v$, and hence
		\begin{equation}
			\label{remainder global}
			\Delta\tau
			-sk_0^2\hat{\xi}_0
			+f_s\sprlr{s\hat{\xi}_0+\tau(s)}
			=0.
		\end{equation}
		We expand $v\mapsto f_s(v)$ at $v\equiv0$ in $\Xcal$ and use the expansion of $M(s)$ in \eqref{expansions squ} to obtain
		\begin{equation}
			\label{global expansion f}
			f_s(v)=0+
			\sprlr{-g+\frac{1}{4}\sprlr{M^\ast(k_0)+\Ocal(\abs{s}^2)}}v+\Ocal(\norm{v}_\Xcal^2)
			=k_0^2v+\Ocal(\abs{s}^2\norm{v}_\Xcal)+\Ocal(\norm{v}_\Xcal^2).
		\end{equation}
		Now define
		\begin{equation*}
			g_s(\tau)\coloneqq
			-sk_0^2\hat{\xi}_0+f_s\sprlr{s\hat{\xi}_0+\tau}
			=
			-sk_0^2\hat{\xi}_0
			+k_0^2\sprlr{s\hat{\xi}_0+\tau}+\Ocal(\abs{s}^2)
		\end{equation*}
		and note that $g_s(\tau(s))=\Ocal(\abs{s}^2)$ in $\Xcal$. 
		This holds as $v(s)=\Ocal(\abs{s})$ in \eqref{global expansion f}.
		We now rewrite \eqref{remainder global} as
		\begin{equation*}
			\Delta\tau(s)+g_s(\tau(s))=0.
		\end{equation*}
		Applying the elliptic higher interior regularity estimate \cite[Sec.~6.3.1, Thm. 2]{evans2010a} on the ball $\Brm_{2\pi/k_0}(0)\supset\overline{\Omega}$, we then obtain
		\begin{equation}
			\label{higher regularity global}
			\norm{\tau(s))}_{\Hrm_\per^m(\Omega)}
			\leq C
			\sprlr{\norm{\tau(s))}_{\Lrm_\per^2(\Omega)}+\norm{g_s(\tau(s))}_{\Hrm_\per^{m-2}(\Omega)}},
		\end{equation}
		where we use periodicity to replace the norms by their periodic counterparts.
		An expansion of 
		\begin{equation*}
			g\colon\Hrm_\per^{m-2}(\Omega)\times\R\longrightarrow\Hrm_\per^{m-2}(\Omega),\qquad
			(\tau,s)\longmapsto g_s(\tau(s))
		\end{equation*}
		at $(\tau,s)=(0,0)$ combined with the expansion of $M(s)$ yields
		\begin{align}
		\label{order of g}
			\normlr{g_s(\tau(s))}_{\Hrm_\per^{m-2}(\Omega)}&=
			\normlr{0+
				\sprlr{\frac{M^\ast}{4}-g}\tau(s)
				+\sprlr{
					\sprlr{\frac{M^\ast}{4}-g-k_0^2}\hat{\xi}_0
					+\sprlr{\frac{M^\ast}{4}-g}\tau'(s)
				}s}_{\Hrm_\per^{m-2}(\Omega)}
			+\Ocal(\abs{s}^2)
			\notag\\
			&=
			\normlr{
				k_0^2\tau(s)
				+k_0^2\tau'(s)
				s}_{\Hrm_\per^{m-2}(\Omega)}
			+\Ocal(\abs{s}^2)
			=\Ocal(\abs{s}^2).
		\end{align}
		Bootstrapping \eqref{higher regularity global} implies the claim.
		
		The same proof works in the case $\Omega=\hexagon$.
		Despite $M(s)$ having a linear term in this case, \eqref{order of g} remains true with $\hat{\xi}_0$ replaced with $\hat{\psi}_0$.
		The reason is that the additional term is cancelled by the corresponding term in the expansion of $M(s)K(s\hat{\psi}_0+\tau(s))$.
	\end{proof}

    We employ the estimate established in \cref{lemma higher tau} in the proof of the next result: 
    the global bifurcation branch is contained in a closed cone in $\Xcal\times\R$, which implies that the curve does not form a closed loop. Another direct consequence of this result is that the solutions along the bifurcation branch satisfy a nodal property and have a certain monotonicity, see \cref{remark nodal}. For the proof, we follow the strategy in \cite{healey1991}, which relies on an application of Hopf's lemma on a subdomain of the fundamental domain.
    
	\begin{proposition}
		\label{prop not loop}
		The global bifurcation branch obtained in \cref{thm glob bif} is not a closed loop.
	\end{proposition}
	
		\begin{figure}[h!]
		\centering
		\includegraphics[width=0.35\textwidth]{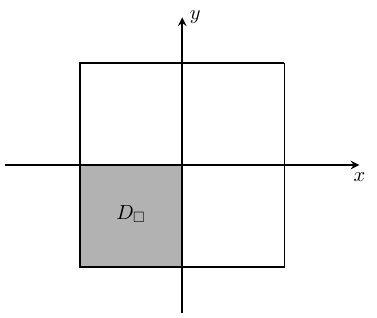}\hspace{2cm}
        \includegraphics[width=0.35\textwidth]{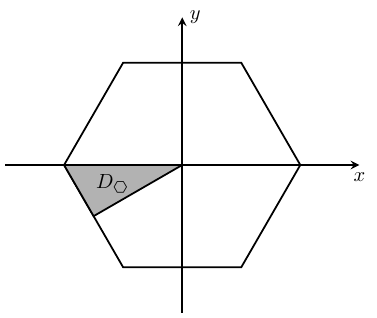}
		\caption{Subdomains $D_{\square}$ (left) and $D_\hexagon$ (right) of the fundamental domain on which $v$ is monotonous.}
		\label{fig:slicemono}
	\end{figure}
	
	\begin{proof}
		We employ \cite[Thm. 9.2.2]{buffoni2003} on global bifurcation in cones. Therefore, we define for the square case
        \begin{equation*}
            \Kcal_\square \coloneqq \set{v \in \Xcal}{\partial_xv,\;\partial_yv\geq 0
				\quad\textrm{in}\quad D_\square \coloneqq \sprlr{-\pi/k_0,0}\times\sprlr{-\pi/k_0,0}}
        \end{equation*}
        and for the hexagonal case 
        \begin{equation*}
            \Kcal_\hexagon \coloneqq \set{v \in \Xcal}{\partial_xv,\;\sqrt{3}\partial_xv+\partial_yv\geq 0
				\quad\textrm{in}\quad  D_\hexagon},
        \end{equation*}
        where $D_\hexagon$ is the triangular subset of the hexagon defined as
        \begin{equation*}
            D_\hexagon \coloneqq \set{(x,y) \in \hexagon}{\pi < \theta < \tfrac{7}{6}\pi,\textrm{ where}\;x + \ii y = r \ee^{\ii\theta}},
        \end{equation*}
        see \cref{fig:slicemono}. Note that both $\Kcal_\square$ and $\Kcal_\hexagon$ are closed cones in the respective function space $\Xcal$.
		
		Again, we restrict the proof to the case $\Omega=\square$, as the other case is analogous.
        To apply \cite[Thm. 9.2.2]{buffoni2003}, we verify the following conditions:
		\begin{enumerate}[label=(\alph*)]
			\item\label{cone cond a} 
			$\Kcal_\square$ is a cone in $\Xcal$;
			\item\label{cone cond b} 
			for some $\varepsilon >0$ the local bifurcation branch satisfies $v(s)\in\Kcal_\square$ for $0\leq s<\varepsilon$;
			\item\label{cone cond c} 
			if $\xi\in\ker\partial_vF(0,M)\cap\Kcal_\square$ for some $M\in\R$, then $\xi=\alpha\xi_0$ for some $\alpha\geq0$ and $M=M^\ast(k_0)$;
			\item\label{cone cond d}
			$\set{(v,M)\in \Scal}{v\nequiv 0} \cap (\Kcal_\square\times\R)$ is open in $\Scal$.
		\end{enumerate}
		Recall that $\xi_0$ is defined in \eqref{kernel element squ} and $\Scal$ is defined in \ref{prop global fredholm} in the proof of \cref{thm glob bif}. Since condition \ref{cone cond a} holds by definition of $\Kcal_\square$, it remains to check the conditions \ref{cone cond b}--\ref{cone cond d}.
		
		\noindent\textbf{Step 1:} Condition \ref{cone cond c}.
		Let $\ker\partial_vF(0,M)\cap\Kcal_\square\neq\{0\}$ and $\xi\in\ker\partial_vF(0,M)\cap\Kcal_\square\setminus\{0\}$.
		In the proof of \cref{thm glob bif} we have computed $\partial_v F(v,M)$ and setting $v\equiv0$ yields 
		\begin{equation*}
			L_{0,M}\xi
			\coloneqq
			\partial_v F(0,M)\xi
			=
			\Delta \xi + \sprlr{\frac{M}{4}-g} \xi,
		\end{equation*}
		as $\xi$ has vanishing mean.
		The Fourier symbol is
		\begin{equation*}
			\hat{L}_{0,M}
			(\bfk)
			= 
			-\abslr{\bfk}^2+\frac{M}{4}-g
			,\qquad 
			\bfk
			\in\Gamma.
		\end{equation*}
		Note that $L_{0,M}$ acts diagonally on Fourier modes, and upon rewriting $\xi$ as a cosine series, we conclude by $\xi\in\Kcal_\square$ that $\xi=\alpha\xi_0$ for some $\alpha\geq0$. Therefore, $L_{0,M}\xi = 0$ if and only if $\hat{L}_{0,M}(\bfk)=0$ for $\bfk \in \Lambda_\square$, where $\Lambda_\square$ is defined in \Cref{eq:def-Lambda}. We already know that the latter is true if and only if $M=M^\ast(k_0)$.
		
		\noindent\textbf{Step 2:} Openness condition \ref{cone cond d}.
		We prove that $\set{(v,M)\in \Scal}{v\nequiv 0} \cap (\Kcal_\square\times\Rbb)$ is open in $\Scal$.
		Let $(v,M)\in\Scal \cap (\Kcal_\square\times\Rbb)$.
		As $F(v,M)=0$ and $v$ is smooth, we may take take the gradient of the equation, which eliminates the nonlocal term. 
		This yields that $v$ solves
		\begin{equation*}
			\nabla\Delta v-g\nabla v
			+
			M\frac{1}{(1+v)(2+v)^2}\nabla v=0.
		\end{equation*}
		In particular, $\partial_xv$ and $\partial_yv$ solve the linear equation
		\begin{equation*}
			\Delta u-g u
			+
			M\frac{1}{(1+v)(2+v)^2} u = 0.
		\end{equation*}
		A refined version of Hopf's Lemma \cite[Sec.~9.5.2, Lem.~1 (ii)]{evans2010a} implies either $\partial_xv\equiv 0$ or $\partial_xv>0$ and simultaneously either $\partial_yv\equiv 0$ or $\partial_yv>0$.
		By $\square$-symmetry, it is not possible that exactly one partial derivative vanishes identically.
		So either $v\equiv0$, or $(v,M)$ is an interior point of $\set{(v,M)\in \Scal}{v\nequiv 0} \cap (\Kcal_\square\times\Rbb)$ in $\Scal$.
		
		\noindent\textbf{Step 3:} Condition \ref{cone cond b}. Recalling the expansion \eqref{expansions squ} of the local bifurcation branch, we prove that for some $\varepsilon >0$ 
		\begin{equation*}
			v(s) = s(\cos(k_0x)+\cos(k_0y))/\normlr{\xi_0}_\Xcal+\tau(s)\in\Kcal_\square\quad\text{for all}\;0\leq s<\varepsilon,
		\end{equation*}
		where $\tau=\Ocal(\abs{s}^2)$ in $\Xcal$.
		Therefore, it is enough to prove that
		\begin{align*}
			\partial_xv(s)=-sk_0 \sin(k_0x)/\normlr{\xi_0}_\Xcal + \partial_x\tau(s) \geq& 0,\\
			\partial_yv(s)=-sk_0 \sin(k_0y)/\normlr{\xi_0}_\Xcal + \partial_y\tau(s) \geq& 0
		\end{align*}
        for all $(x,y) \in D_\square$ and $0 \leq s < \varepsilon$.
		By symmetry, it suffices to prove the first inequality.
		Fix $y\in[-\pi/k_0,0]$ and note that by compactness we can choose $\varepsilon$ uniformly in $y$.
		If $y\in\{-\pi/k_0,0\}$ there is nothing to prove, so assume $y\in(-\pi/k_0,0)$.
		
		We consider $\Tilde{v}(s)\coloneqq \normlr{\xi_0}_\Xcal v(s)/s$ and $\Tilde{\tau}(s)\coloneqq \normlr{\xi_0}_\Xcal\tau(s)/s$ ($=\Ocal(\abs{s})$ in $\Xcal$), and we prove the equivalent inequality
		\begin{equation}
			\label{needed ineq global}
			\partial_x\Tilde{v}(s)=
			-k_0 \sin(k_0x) + \partial_x\Tilde{\tau}(s) \geq 0\quad\text{for all}\;
			x\in(-\pi/k_0,0)
			\;\text{and}\;
			0\leq s<\varepsilon.
		\end{equation}
        
		It suffices to prove that there is a $\kappa>0$ such that
		\begin{equation}
			\label{claim ineq global}
			\partial_{xx}\Tilde{v}(s)\vert_{x=-\pi/k_0}\geq\kappa,
			\qquad
			\partial_{xx}\Tilde{v}(s)\vert_{x=0}\leq-\kappa 
			\qquad\text{for all}\;
			0\leq s<\varepsilon.
		\end{equation}
		Indeed, as $v$ is $\Omega$-symmetric and $\Omega$-periodic, we know that $\partial_{x}\Tilde{v}(s)\vert_{x=-\pi/k_0}=\partial_{x}\Tilde{v}(s)\vert_{x=0}=0$.
		By elliptic regularity $\Tilde{v}(s)$ is smooth and thus \eqref{claim ineq global} implies the existence of $\delta>0$ such that
		\begin{equation*}
			\partial_x\Tilde{v}(s)\geq 0,
			\qquad
			\text{on}\;
			[-\pi/k_0,-\pi/k_0+\delta]
			\cup
			[-\delta,0]
			\quad
			\text{for all}\;
			0\leq s<\varepsilon
		\end{equation*}
		and
		\begin{equation*}
			\partial_{x}\Tilde{v}(s)\vert_{x=-\pi/k_0+\delta}\geq\frac{1}{4}\delta\kappa>0,
			\qquad
			\partial_{x}\Tilde{v}(s)\vert_{x=-\delta}\geq\frac{1}{4}\delta\kappa>0 
			\qquad\text{for all}\;
			0\leq s<\varepsilon.
		\end{equation*}
		To see that $\delta$ is $s$-independent, it suffices to expand $\Tilde{v}(s)$ and use \cref{lemma higher tau} along with the Sobolev embedding theorem.
		On $[-\pi/k_0+\delta,\delta]$ we can find a bound $-k_0 \sin(k_0x)\geq c>0$.
		Again, by \cref{lemma higher tau} and by the Sobolev embedding theorem, it follows that $\Tilde{\tau}(s)=\Ocal(|s|)$ in $\Crm^1_\per(\Omega)$ and, upon possibly choosing a smaller $\varepsilon>0$, we obtain \eqref{needed ineq global}.
		
		It therefore remains to prove \eqref{claim ineq global}.
		Note that $\partial_{xx}\Tilde{v}(s)=-k_0^2\cos(k_0x)+\partial_{xx}\Tilde{\tau}(s)$.
		As $\abs{\cos(k_0x)}=1$ for $x\in\{-\pi/k_0,0\}$, it suffices to show that $\partial_{xx}\Tilde{\tau}(s)=\Ocal(\abs{s})$ in $\Crm_\per^0(\Omega)$.
		This is true by \cref{lemma higher tau} combined with the Sobolev embedding theorem.
	\end{proof}
	
	\begin{remark}
		\label{remark nodal}
		The result \cite[Thm. 9.2.2]{buffoni2003} also ensures that $v(s)\in\Kcal_{\Omega}\setminus\{0\}$, $\Omega\in \{\square,\hexagon\}$ for $s>0$.
		As in the proof of condition \ref{cone cond d}, a refined version of Hopf's lemma implies that $v(s)$ is in the interior of $\Kcal_\Omega$, that is, the monotonicity properties in $D_\Omega$ are strict.
		Therefore, for all $s>0$ the function $v(s)$ attains its only global maximum on $\overline{\Omega}$ at $(x,y)=(0,0)$, while its global minimum on $\overline{\Omega}$ is attained at the corners of the fundamental domain $\overline{\Omega}$.
		The $v(s)$ feature a nodal pattern of minima and maxima, which is invariant along the bifurcation curve.
		This is called a nodal property.
	\end{remark}

	It turns out that alternative \ref{blow up} in \cref{thm glob bif} cannot be caused by a mere blow-up of derivatives.
	In particular, we can exclude corners and cusps as long as $v$ and $M$ remain bounded in $\Lrm^2_\mathrm{per} \times \R$, and $v$ remains bounded away from $-1$.
	
	\begin{proposition}
    \label{prop no derivative blowup}
		We can replace
		alternative \ref{blow up} in \cref{thm glob bif} with
		\begin{enumerate}[label=(\roman*)'']
			\item
			$\normlr{\sprlr{v(s),M(s)}}_{\Lrm_\per^2(\Omega)\times\Rbb}\longrightarrow\infty$ as $s\rightarrow\pm\infty$.
		\end{enumerate}
		
		In particular, a solution $\sprlr{v,M}$ admits the bound		
		\begin{equation}\label{eq:reg-bound}
			\norm{v}_\Xcal
			\lesssim
			g\norm{v}_{\Lrm_\per^2(\Omega)}
			+
			M\sprlr{
			1
			+
			\normlr{\log\sprlr{1+v}}_{\Lrm_\per^2(\Omega)}
			+
			\normlr{1+v}_{\Lrm_\per^2(\Omega)}
			}.
		\end{equation}
	\end{proposition}
	\begin{proof}
		Recall that a solution $v\in\Xcal$ satisfies
		\begin{equation*}
			\Delta v=gv
			+
			M\sprlr{\frac{1}{2+v}+\log\sprlr{\frac{1+v}{2+v}}}
			-MK(v),
		\end{equation*}
		where the constant $K(v)$ is defined in \eqref{eq:K}.
		By the vanishing mean condition, Poincaré's inequality implies $\norm{v}_{\Lrm_\per^2(\Omega)}\lesssim\norm{\nabla v}_{\Lrm_\per^2(\Omega)}$.
		In addition, by $\Omega$-symmetry, $\norm{\nabla v}_{\Lrm_\per^2(\Omega)}\lesssim\norm{\Delta v}_{\Lrm_\per^2(\Omega)}$.
		
		As for the right-hand side, we use $v>-1$ to estimate
		\begin{align*}
			\normlr{\frac{1}{2+v}+\log\sprlr{\frac{1+v}{2+v}}}_{\Lrm_\per^2(\Omega)}
			\leq&
			\abslr{\Omega}^\frac12
			+
			\normlr{\log\sprlr{1+v}}_{\Lrm_\per^2(\Omega)}
			+
			\normlr{\log\sprlr{2+v}}_{\Lrm_\per^2(\Omega)}
			\\\leq&
			\abslr{\Omega}^\frac12
			+
			\normlr{\log\sprlr{1+v}}_{\Lrm_\per^2(\Omega)}
			+
			\normlr{1+v}_{\Lrm_\per^2(\Omega)}.
		\end{align*}
		Lastly, we estimate $K(v)$
		\begin{equation*}
			\abslr{K(v)}
			=
			\abslr{\fint_\Omega\frac{1}{2+v}+\log\sprlr{\frac{1+v}{2+v}}\de x}
			\leq
			1
			+
			\abslr{\Omega}^{-1}\sprlr{\normlr{\log\sprlr{1+v}}_{\Lrm_\per^1(\Omega)}
			+
			\normlr{1+v}_{\Lrm_\per^1(\Omega)}}.
		\end{equation*}
		As $\Omega$ has finite measure, we can bound $\Lrm_\per^1(\Omega)$-norms by $\Lrm_\per^2(\Omega)$-norms.
	\end{proof}
	
	\subsection{A conditional result on film rupture}
    In the one-dimensional case, it is possible to show that only alternative \ref{approach boundary} occurs along the global bifurcation branch. The proof in the one-dimensional case relies on a Hamiltonian structure of the spatial dynamics formulation of the bifurcation problem and thus cannot be transferred to the two-dimensional setting.

    It turns out that the main challenge in the two-dimensional case is to obtain a bound on the Marangoni number $M(s)$ along the bifurcation curve. While this remains an open problem, we prove that alternative \ref{approach boundary} in \cref{thm glob bif} has to occur under the additional assumption that $M(s)$ stays bounded. In \Cref{sec:numerics} we indeed find numerical evidence that this assumption is satisfied along the bifurcation curve: in the square case, $M(s)$ is decreasing similarly to the one-dimensional case, see \Cref{fig:num bif diag squ}. For hexagonal patterns, $M(s)$ is decreasing along the branch of down-hexagons. Along the branch of up-hexagons, there is a fold point, where $M(s)$ attains its maximum. Afterwards, $M(s)$ is also decreasing, see \Cref{fig:num bif diag hex}.
    	
	\begin{proposition}
		\label{prop cond rupture}
		If $M(s)$ is bounded along the global bifurcation branch obtained in \cref{thm glob bif}, then the solutions on this branch exhibit film rupture in the sense that
		\begin{equation*}
			\inf_{s\in\Rbb}\min_{x\in\Omega}v(s)=-1.
		\end{equation*}
	\end{proposition}
	\begin{proof}
		We provide a proof by contradiction.
		Assume there exists $c>-1$ such that $\min_\Omega v(s)>c$ for all $s\in\Rbb$, which in particular implies that assumption \ref{approach boundary} does not occur.
		We prove that $\|v(s)\|_\Xcal$ allows for a bound depending on $M(s)$ and $c$. 
		Together with the condition of $M(s)$ being bounded, this rules out the alternative \ref{blow up} from \cref{thm glob bif}.
		As we have already ruled out the alternative \ref{loop} in \cref{prop not loop}, the assumption that $v(s)$ is bounded away from $-1$ is absurd.
		
		By the nodal property established in \cref{remark nodal}, we know 
		\begin{equation*}
			\max_{(x,y)\in\Omega} v(s)(x,y)=v(s)(0,0)\eqqcolon v_\mathrm{max}(s) 
			\qquad\mathrm{and}\qquad 
			\Delta v(s)(0,0) < 0. 
		\end{equation*}
		In combination with the bifurcation problem \eqref{bif problem}, this gives the inequality
		\begin{equation*}
			-g v_\mathrm{max}(s) + M\left(\frac{1}{2+v_\mathrm{max}(s)} + \log\left(\frac{1+v_\mathrm{max}(s)}{2+v_\mathrm{max}(s)}\right)\right) - M K(v(s)) > 0.
		\end{equation*}
		Hence, we obtain the estimate
		\begin{equation}
			\label{bound v by M}
			v_\mathrm{max}(s) < \frac{M(s)}{g}\left[\left(\frac{1}{2+v_\mathrm{max}(s)} + \log\left(\frac{1+v_\mathrm{max}(s)}{2+v_\mathrm{max}(s)}\right)\right) - K(v(s))\right] < -\frac{M(s) K(v(s))}{g},
		\end{equation}
		where we use that 
		\begin{equation*}
			\frac{1}{2+v_\mathrm{max}(s)} + \log\left(\frac{1+v_\mathrm{max}(s)}{2+v_\mathrm{max}(s)}\right) < 0.
		\end{equation*}
		As we assumed $\min_\Omega v(s)>c>-1$, we can check in \eqref{eq:K} that $K(v(s))$ allows for a bound depending on $c$.
		By \eqref{bound v by M}, along with the assumption that $M(s)$ is uniformly bounded, we find that $\|v(s)\|_{\Lrm^\infty}$ is bounded uniformly in $s \in \R$. 
        This additionally provides a uniform bound on $\|v(s)\|_{\Lrm_\per^2}$ because $\Omega$ has finite measure. 
        As $v(s)$ is also uniformly bounded away from $-1$ by assumption, we obtain that $\normlr{\sprlr{v(s),M(s)}}_{\Xcal\times\Rbb}$ is bounded using the regularity estimate \eqref{eq:reg-bound}.
		As laid out above, this yields a contradiction.
	\end{proof}

    \section{Numerical experiments}\label{sec:numerics}
    This section is devoted to a numerical analysis of the bifurcation problem \eqref{bif problem} for both the square and the hexagonal setups.
    The numerical experiments match our analytical results from \cref{sec loc bif,sec glob bif} and provide additional insight:
    \begin{enumerate}
        \item We detect bifurcation branches of (periodic and symmetric) square and hexagon solutions predicted by \Cref{thm loc bif squ,thm loc bif hex,thm glob bif}.
        \item All numerical continuations of nontrivial bifurcation branches break down at a solution that is close to film rupture. 
        Along these numerical continuations, the bifurcation parameter $M(s)$ remains uniformly bounded, which justifies the conditional result on film rupture in \cref{prop cond rupture}.
       \item A period-doubling secondary bifurcation occurs in the square case.
    \end{enumerate}
    
    The numerical continuation of branches of solutions to the bifurcation problem \eqref{bif problem} is conducted in \textsc{Matlab} using the \textsc{pde2path} library, see \cite{dohnal,uecker2014,uecker2021}.
    In each continuation step, the nonlinear partial differential equation \eqref{eq:integrated-stationary-problem} is solved numerically using a finite element method on a rectangular domain $\Omega^\textrm{num}$ with homogeneous Neumann boundary conditions.
    Mass conservation is enforced by introducing a Lagrange parameter $\lambda$, which replaces the nonlocal term $K(v)$. In addition, we set $g = 1$ and therefore $M^\ast=4g=4$. 

    In each experiment, first we numerically continue the trivial branch starting in $\sprlr{v,M}=\sprlr{0,M^\textrm{init}}$ for a predetermined number of steps.
    Possible bifurcation points are detected and stored by \textsc{pde2path}.
    In a second step, we choose a bifurcation point and continue the numerical branch in the direction of a kernel element. Note that since symmetry is not enforced, the kernel at the bifurcation point on the trivial branch is two-dimensional both in the square and hexagonal case. Therefore, to ensure $\Omega$-symmetry of solutions, we need to choose an $\Omega$-symmetric kernel element.
    In \cref{fig:num bif diag squ,fig:num bif diag squ20,fig:num bif diag hex}, the curves are branches of nontrivial solutions obtained numerically.

    Possible secondary bifurcation points are detected along the primary bifurcation branches, which we discuss in further detail in \cref{sec num squ,sec num hex}.
    In all numerical experiments, the continuation of nontrivial (primary or secondary) bifurcation branches breaks down at a solution close to film rupture.
    
    \subsection{Numerical results on the square lattice}
    \label{sec num squ}
    To detect branches of square pattern solutions, we choose the square domain $\Omega^\textrm{num}_\square=[-2\pi,2\pi]\times[-2\pi,2\pi]$. The domain is discretised using a criss-cross mesh with 100 discretisation points per dimension, leading to a mesh with $4\times99^2$ elements.
    The maximal step sizes in $s$ are 0.03 on the trivial branch and 0.01 on the other two branches.
    
    The domain $\Omega^\textrm{num}_\square$ corresponds to an absolute wave number $k_0=\tfrac{1}{2}$ and to better illustrate the numerical solutions we study the bifurcation point at $M^\ast(2k_0) = 4g + 4(2k_0)^2=8$, such that two periods per dimension fit into $\Omega^\textrm{num}_\square$ (cf. \cref{rema other bifurcation points}).
    Hence we start the continuation of the trivial branch at $\sprlr{v,M^\textrm{init}}=(0,7.9)$.
    
    The numerical bifurcation branch is displayed in \cref{fig:num bif diag squ}.
    In qualitative accordance with the analytic expansion \eqref{expansions squ}, a subcritical pitchfork bifurcation occurs at $(v,M)=(0,8)$.
    \begin{figure}[h!]
		\centering
        \includegraphics[width=0.7\textwidth]{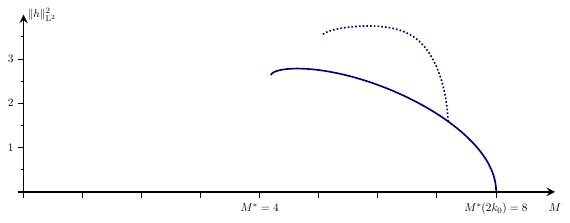}
		\caption{Numerical bifurcation diagram for square patterns on $\Omega_\square^{\textrm{num}}$. The dotted line shows the secondary bifurcation.}
		\label{fig:num bif diag squ}
	\end{figure}
    
    In Figure \ref{fig:squ bif and rup}, we plot two solutions on the bifurcation branch: First, we plot the numerical solution after five continuation steps, which is still close to the bifurcation point $(0,8)$, that is, $M(5) = 7.9217$ and $\normlr{v(5)}_{\Lrm^2} = 0.5032$.
    In addition, the solution after 206 steps is displayed.
    At this point, the solution is close to film rupture, that is, $\min_{\Omega^{\textrm{num}}} v=-0.9933$, and the numerical continuation terminates since the vector field degenerates at $v = -1$.
    \begin{figure}[H]
		\centering
		\includegraphics[height=4cm]{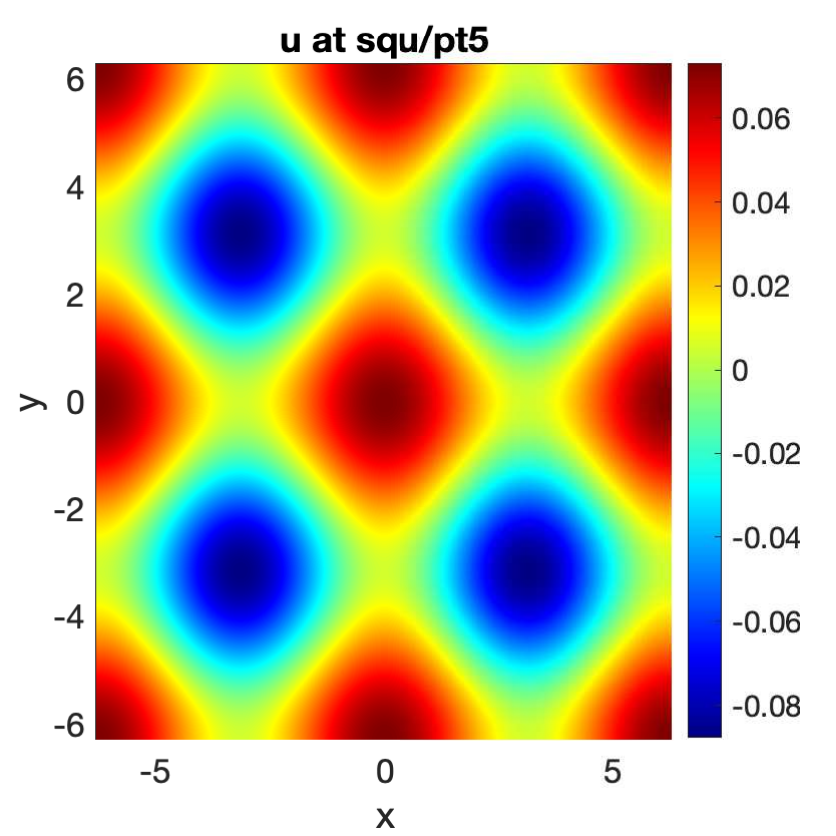}
		\hspace{1cm}
		\includegraphics[height=4cm]{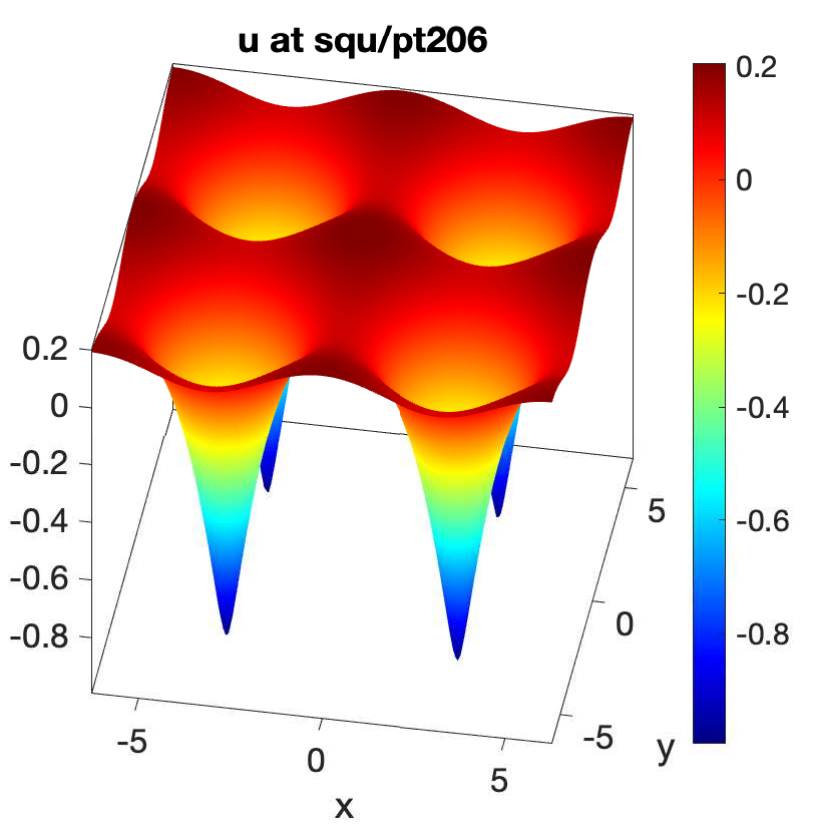}
        \caption{Square pattern solution close to primary bifurcation (left) and close to film rupture (right).}
        \label{fig:squ bif and rup}
    \end{figure}

    Along the primary bifurcation branch, five possible secondary bifurcation points are detected, but we find that only at the third one, an admissible (periodic and symmetric) solution bifurcates.
    In \cref{fig:squ sec bif}, the solution at the secondary bifurcation point is displayed, as well as after 10, 25 and 120 continuation steps on the new branch.
    The bifurcating numerical solution is symmetric and periodic with respect to a fundamental domain with double edge length. Therefore, this indicates a period-doubling bifurcation.
    Once more, the continuation of the secondary branch terminates close to film rupture ($\min_{\Omega^{\textrm{num}}} v=-0.9925$).
    \begin{figure}[h!]
		\centering
		\includegraphics[height=4cm]{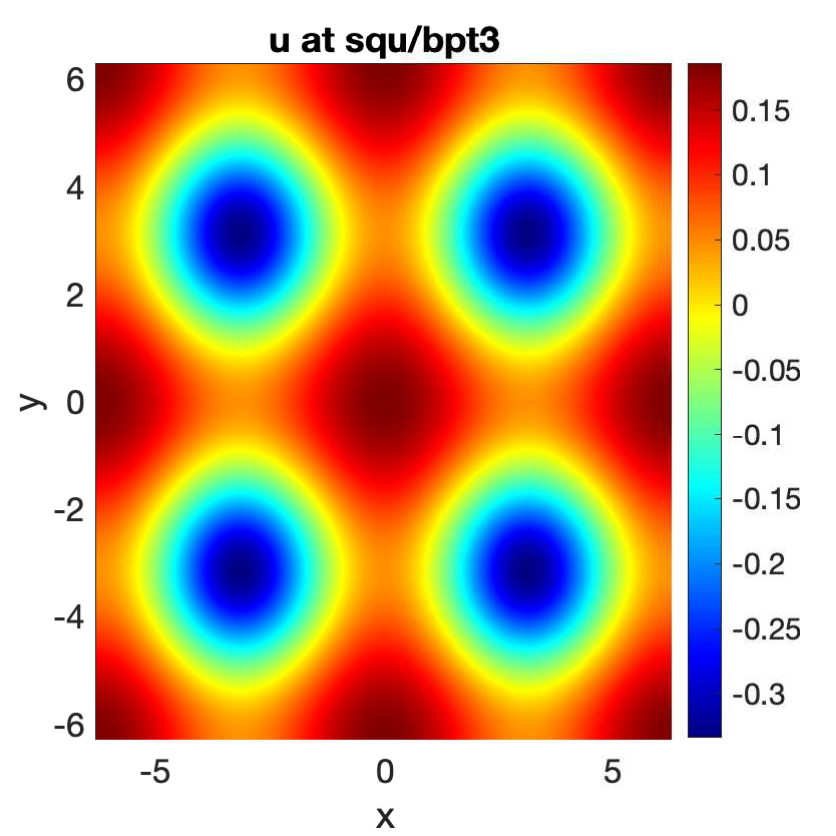}
        \hspace{0.1cm}
		\includegraphics[height=4cm]{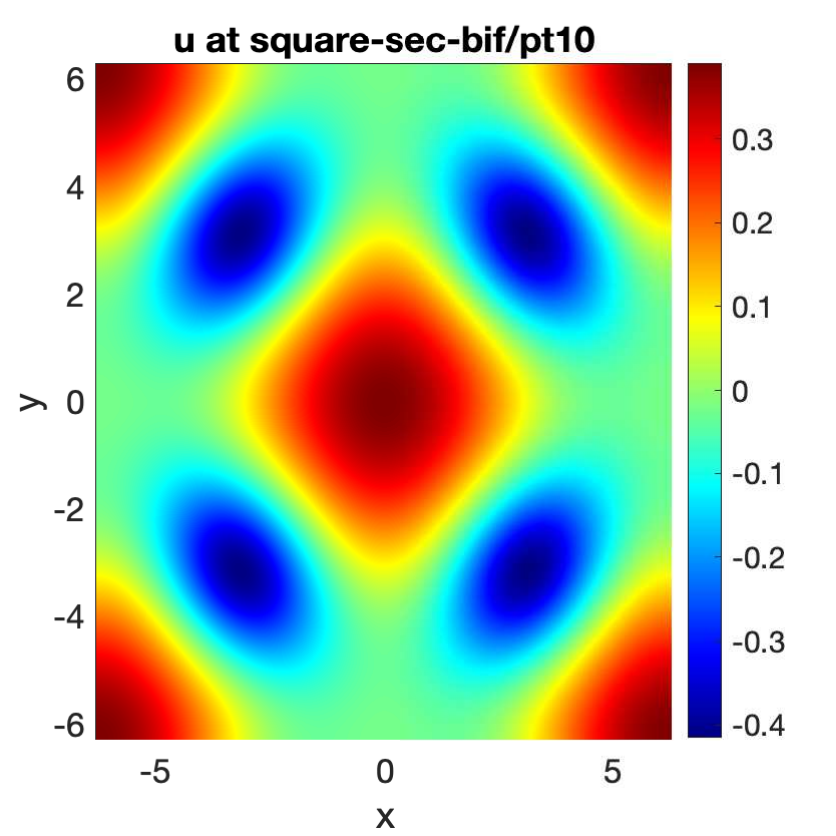}
        \hspace{0.1cm}
		\includegraphics[height=4cm]{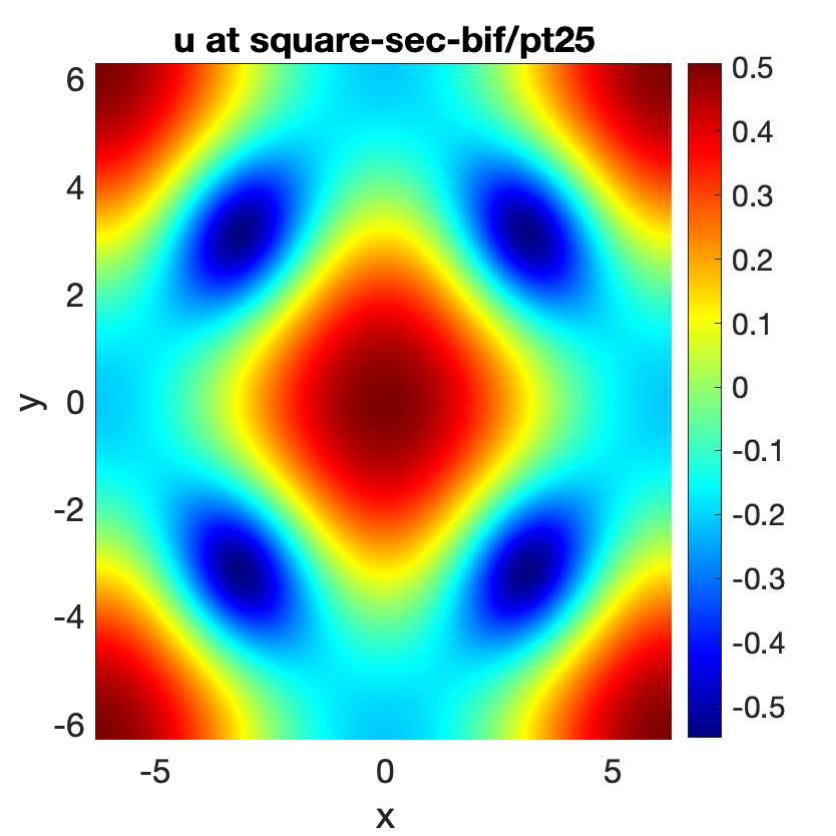}
        \hspace{0.1cm}
		\includegraphics[height=4cm]{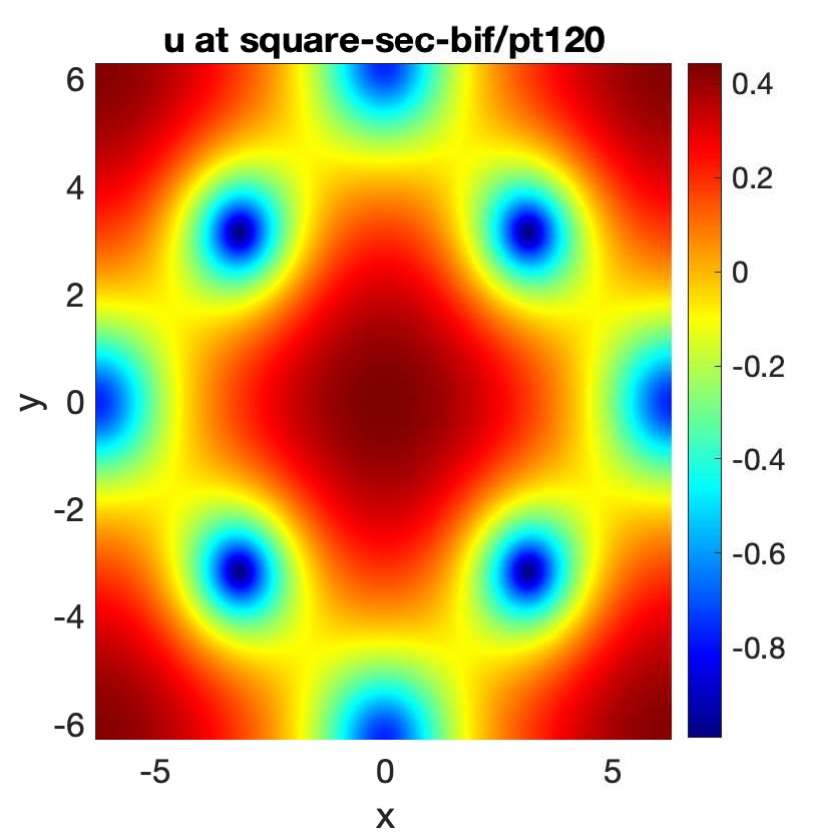}
        \caption{Square pattern solution along the secondary bifurcation branch, from the secondary bifurcation point (left) to almost film rupture (right).}
        \label{fig:squ sec bif}
	\end{figure}

    In \cref{fig:squ plots}, the minimal value, the Lagrange parameter $\lambda=K(s)$, and the $\Lrm^2$-norm of the dominant logarithmic term $\log(1+v)$ are plotted.
    The latter seems to remain bounded, which indicates a uniform $\Hrm^2$-bound for $v(s)$. In the one-dimensional case, this can be proved analytically and can be used to pass to the limit along a subsequence to obtain a weak, stationary film-rupture solution to \eqref{eq:thin-film}, see \cite[Sec.~5]{bruell2024}. However, in the two-dimensional case, we cannot obtain a uniform bound on the $\Lrm^2$-norm of the logarithm analytically. Although the numerical results in \Cref{fig:squ plots} suggest that such a bound exists, we point out that this observation might not be conclusive since the logarithm only becomes singular very close to $v = -1$. For a more detailed discussion, we refer to \Cref{sec:discussion}.
    \begin{figure}[h!]
		\centering
        \includegraphics[width=0.7\textwidth]{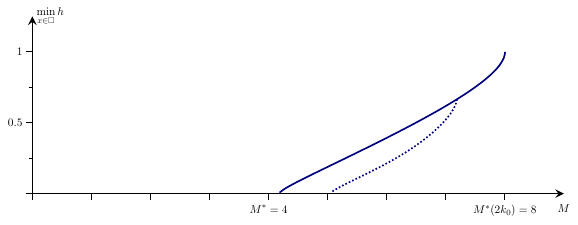}
        
        \includegraphics[width=0.7\textwidth]{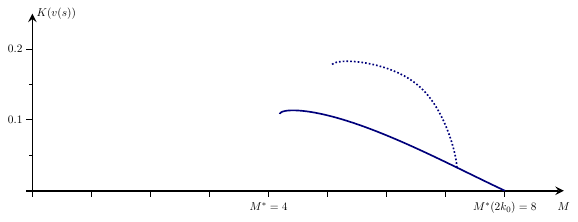}

        \includegraphics[width=0.7\textwidth]{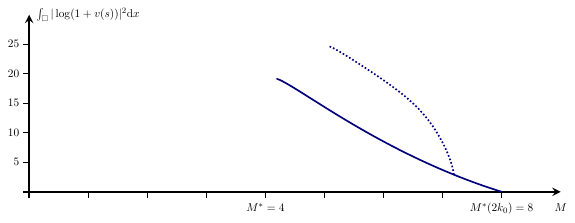}
        \caption{Plots of the minimal value, the constraint parameter $K(v(s))$, and the $\Lrm^2$-norm of the dominant logarithmic term for the primary and secondary bifurcation branch.}
        \label{fig:squ plots}
	\end{figure}

    We have also analysed the primary branch bifurcating at $(0,M^*(3k_0)) = (0,13)$. 
    For this experiment, we have set 200 discretisation points per dimension.
    However, we have not detected any admissible secondary bifurcations that respect periodicity and symmetry.
    In contrast, numerical continuation of the branch bifurcating at $(0,M^*(4k_0)) = (0,20)$ with 200 discretisation points per dimension produces two admissible secondary bifurcations.
    As expected, the primary bifurcation branch as well as the secondary branches terminate close to film rupture, see \cref{fig:num bif diag squ20,fig:squ20 sec bif}.
    By analogous reasoning as above, we compare the edge lengths of fundamental domains, which indicate that the first admissible secondary bifurcation is period-quadrupling, whereas the second one is period-doubling, see the central and right images in \cref{fig:squ20 sec bif}.
    
    \begin{figure}[h!]
		\centering
        \includegraphics[width=0.7\textwidth]{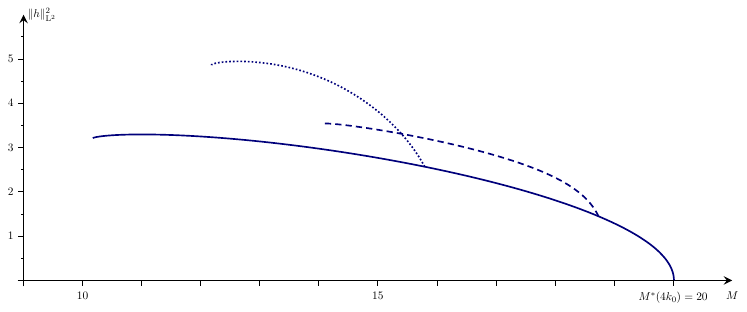}
		\caption{Numerical bifurcation diagram on $\Omega_\square^{\textrm{num}}$ for a square pattern bifurcating at $M^*(4k_0)=20$ with two admissible secondary bifurcations. The dashed line is the first admissible secondary bifurcation, the dotted line is the second admissible secondary bifurcation.}
		\label{fig:num bif diag squ20}
	\end{figure}
    
    \begin{figure}[h!]
		\centering
		\includegraphics[height=4cm]{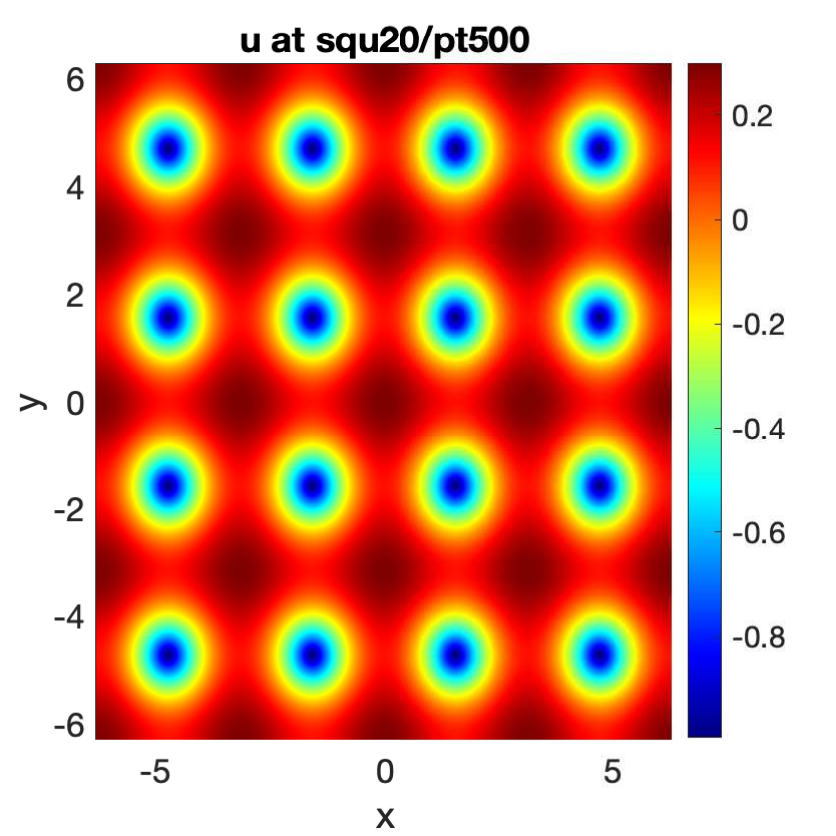}
        \hspace{0.1cm}
		\includegraphics[height=4cm]{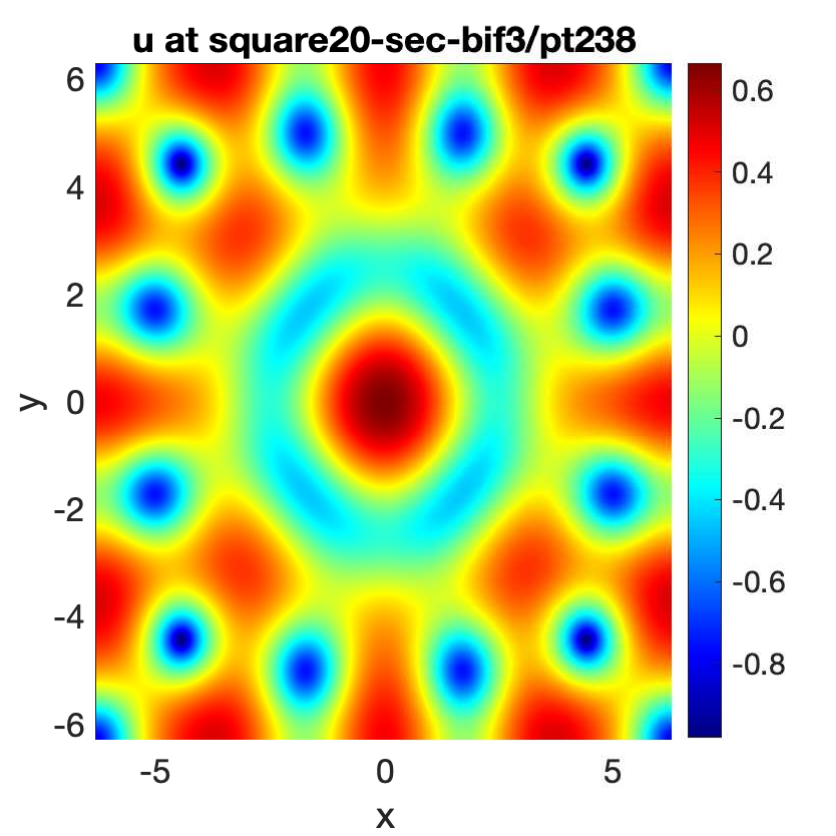}
        \hspace{0.1cm}
		\includegraphics[height=4cm]{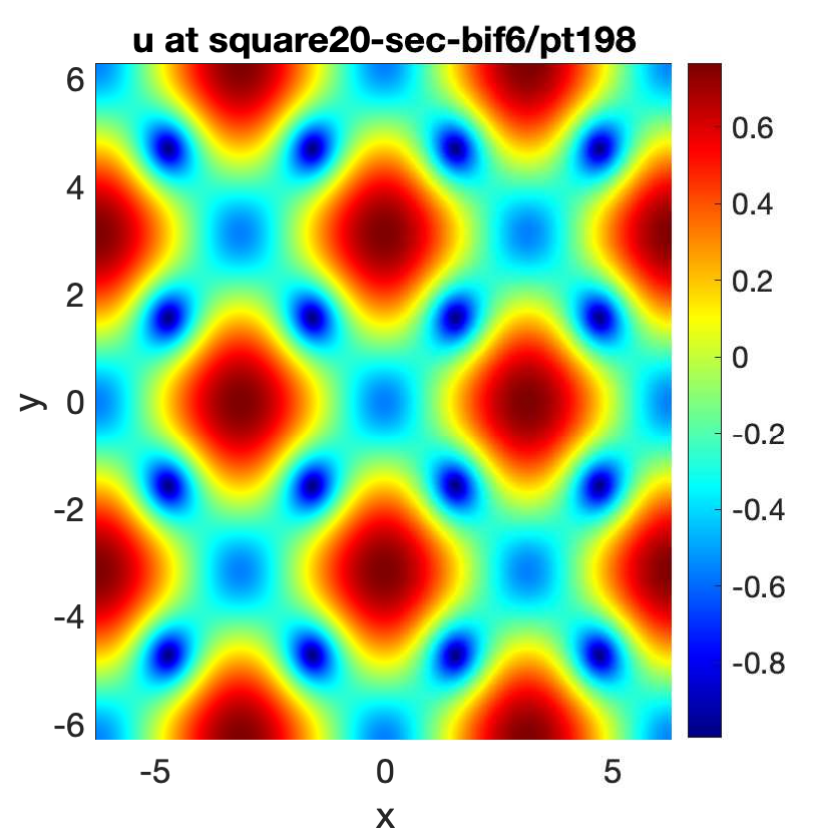}
        \caption{Square pattern solutions close to film rupture on the primary branch (left), and on the first and second admissible secondary branches (centre and right).}
        \label{fig:squ20 sec bif}
	\end{figure}
    
    \subsection{Numerical results on the hexagonal lattice}
    \label{sec num hex}
    Since \textsc{pde2path} employs a finite element method for rectangular domains, we need to choose such a domain in a way such that solutions with hexagonal periodicity and symmetry can appear on this domain, subjected to homogeneous Neumann boundary conditions.
    We choose $\Omega^\textrm{num}_\hexagon=[-2\pi,2\pi]\times[-\tfrac{2\pi}{\sqrt{3}},\tfrac{2\pi}{\sqrt{3}}]$. Again, we use 100 discretisation points per dimension and a criss-cross mesh.
    Since $\Omega^\textrm{num}_\hexagon$ is smaller than $\Omega^\textrm{num}_\square$ by the factor $\tfrac{1}{\sqrt{3}}$, this results in a much higher resolution. 
    The maximal step sizes in $s$ are again 0.03 on the trivial branch and 0.01 on the other two branches.

    In general, a higher resolution is needed to study the bifurcation of hexagon solutions.
    At lower resolution, we observe that one bifurcation point with a two-dimensional kernel splits up into two bifurcation points with one-dimensional kernels $\psi_1 = \cos(\bfk_1\cdot \x) = \cos(k_0x)$ and $\psi_2 = \cos(\bfk_2 \cdot \x) + \cos((\bfk_1 + \bfk_2) \cdot \x)$, respectively. In this situation, one cannot combine kernel elements to obtain the direction of hexagon solutions anymore.
    
    The domain $\Omega_\hexagon^{\mathrm{num}}$ corresponds to $k_0=1$ in the sense that this is the smallest rectangle that contains one fundamental domain of a hexagonal pattern of size $\tfrac{4\pi}{\sqrt{3}k_0}$ with $k_0=1$ and the hexagonal pattern respects the homogeneous Neumann boundary condition, see Figure~\ref{fig:num_domain_hex}.
    Therefore, we study the bifurcation point $M^\ast(k_0)=8$ and set the continuation of the trivial branch to start at $\sprlr{v,M^\textrm{init}}=(0,7.9)$.

    \begin{figure}[h!]
		\centering
        \includegraphics[height=4cm]{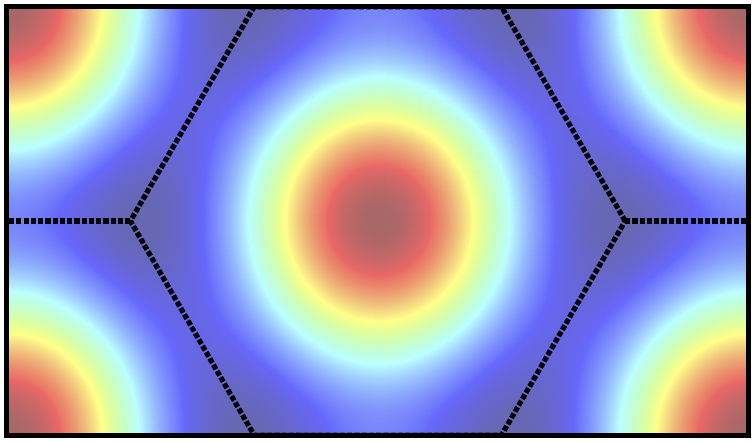}
		\caption{The domain $\Omega^\textrm{num}_\hexagon=[-2\pi,2\pi]\times[-\tfrac{2\pi}{\sqrt{3}},\tfrac{2\pi}{\sqrt{3}}]$ for the numerical analysis on the hexagonal lattice is chosen such that solutions with hexagonal periodicity and symmetry satisfy homogeneous Neumann boundary conditions.}
		\label{fig:num_domain_hex}
	\end{figure}

    The numerical bifurcation branch is displayed in \cref{fig:num bif diag hex}. Since the bifurcation is transcritical, see \Cref{thm loc bif hex}, two branches emerge.
    The branch with $M(s) - M^*(k_0) > 0$ consists of up-hexagons, and thus each maximum is surrounded by six minima.
    The other branch with $M(s) - M^*(k_0) < 0$ consists of down-hexagons, where each minimum is surrounded by six maxima, see \Cref{fig:hex bif up down}.
    \begin{figure}[h!]
		\centering
        \includegraphics[width=0.7\textwidth]{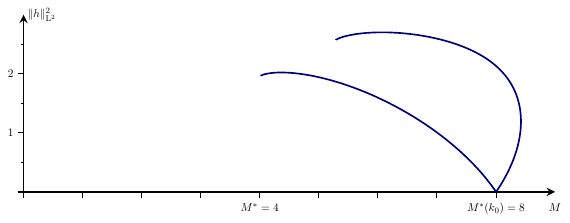}
		\caption{Numerical bifurcation diagram for a hexagonal pattern. The branch with the fold consists of up-hexagons, the other of down-hexagons.}
		\label{fig:num bif diag hex}
	\end{figure}
    \begin{figure}[h!]
		\centering
		\includegraphics[height=4cm]{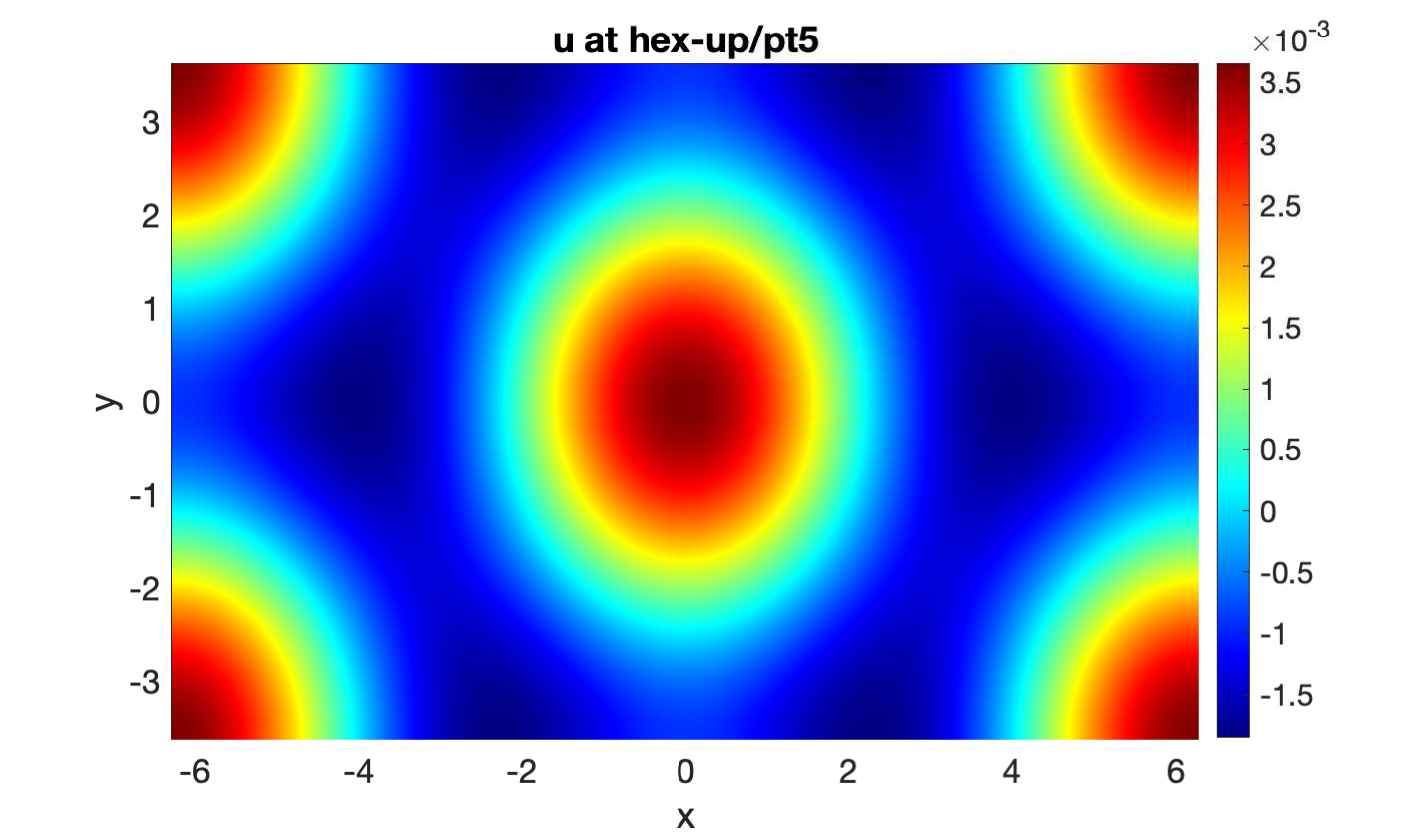}
		\hspace{1cm}
		\includegraphics[height=4cm]{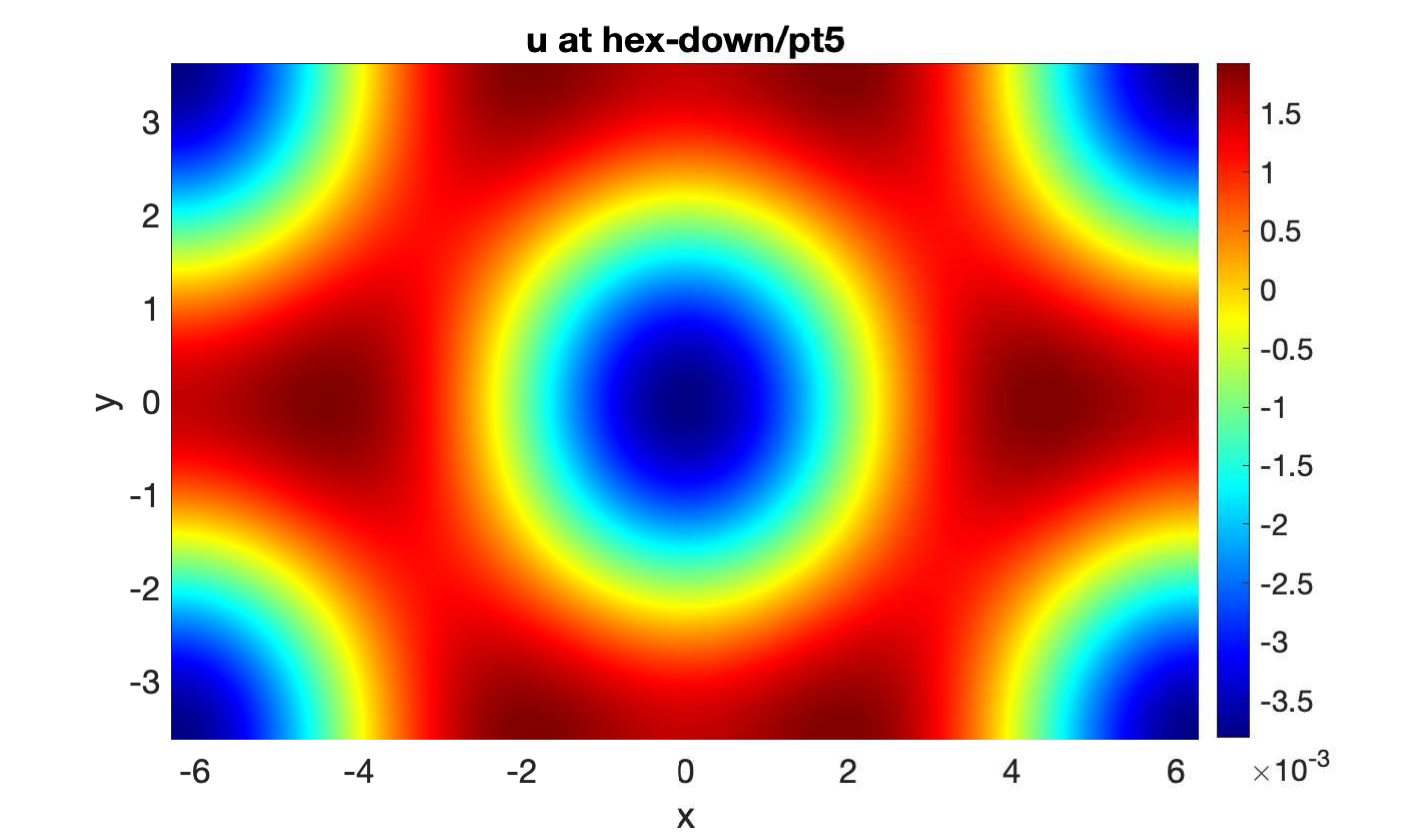}
        \caption{Up-hexagon solution (left) and down-hexagon solution (right) close to bifurcation.}
        \label{fig:hex bif up down}
    \end{figure}

    As in the square case, we plot the solutions after $5$ continuation steps in \Cref{fig:hex bif up down}.
    In addition, the up-hexagon solution after 224 steps, and the down-hexagon solution after 275 steps are displayed in \Cref{fig:hex rup}.
    These solutions are close to film rupture with $\min_{\Omega^{\textrm{num}}} v=-0.9887$ for up-hexagons and $\min_{\Omega^{\textrm{num}}} v=-0.9823$ for down-hexagons, and the numerical continuation terminates there.
    \begin{figure}[h!]
		\centering
		\includegraphics[height=4.75cm]{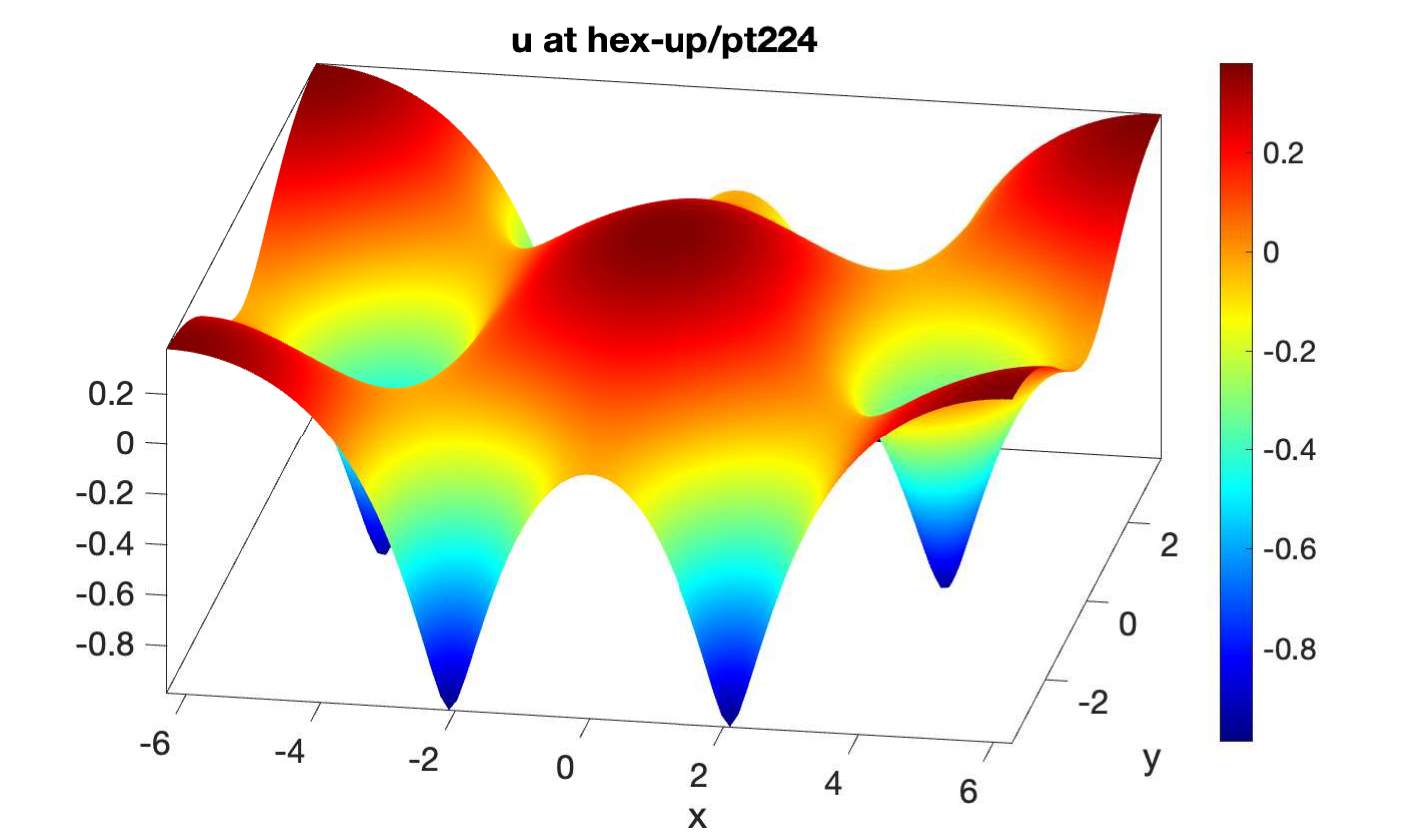}
        \hspace{0.5cm}
        \includegraphics[height=4.75cm]{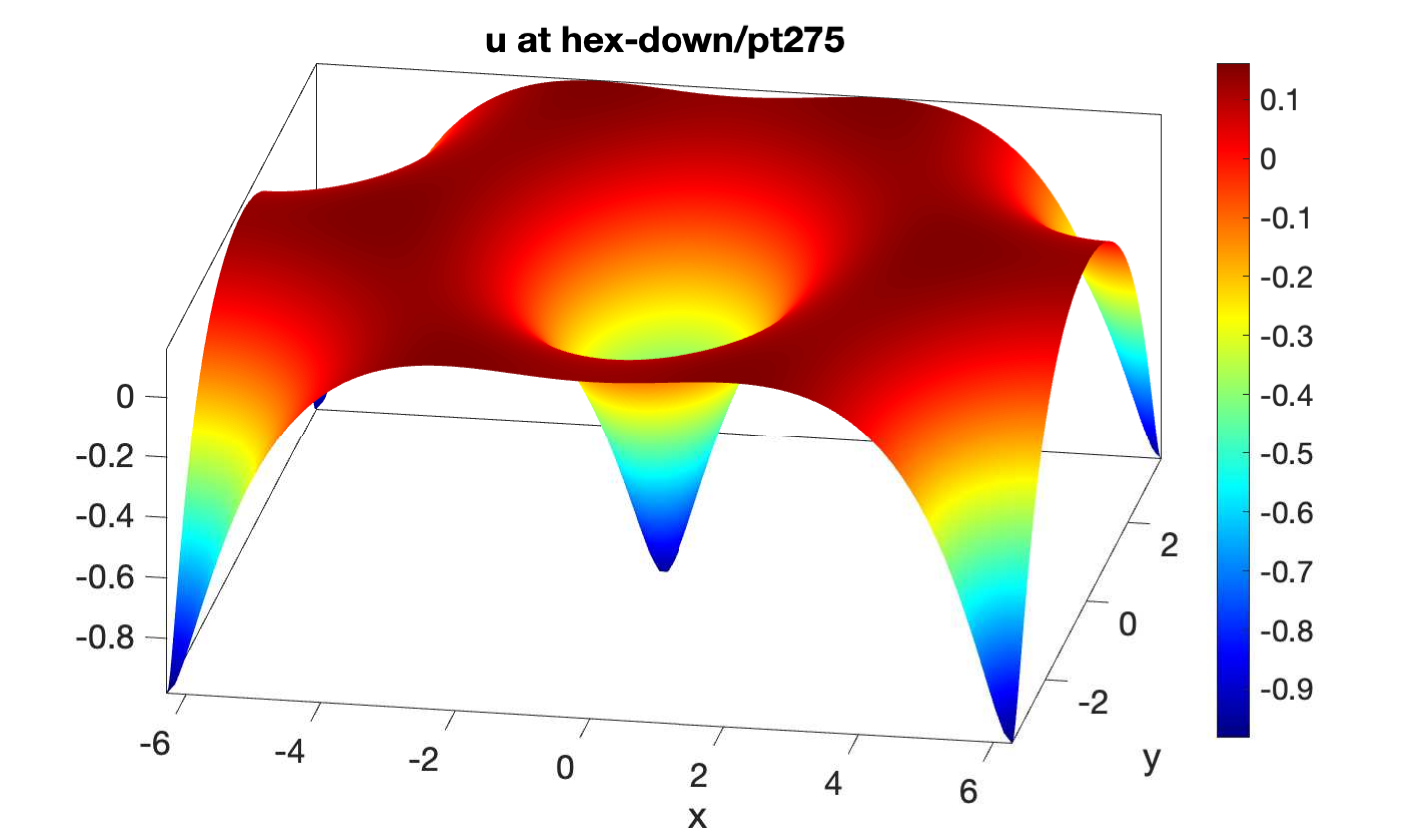}
        \caption{Up-hexagon solution (left) and down-hexagon solution (right) close to film rupture.}
        \label{fig:hex rup}
    \end{figure}

    In \cref{fig:hex plots}, the minimal value, the Lagrange parameter $\lambda=K(s)$, and the $\Lrm^2$-norm of the dominant logarithmic term $\log(1+v)$ are plotted.
    Again, note that the latter remains bounded.

    \begin{figure}[h!]
		\centering
        \includegraphics[width=0.7\textwidth]{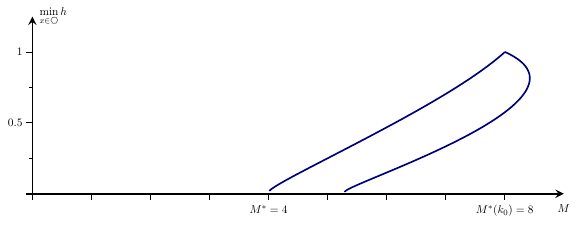}
        
        \includegraphics[width=0.7\textwidth]{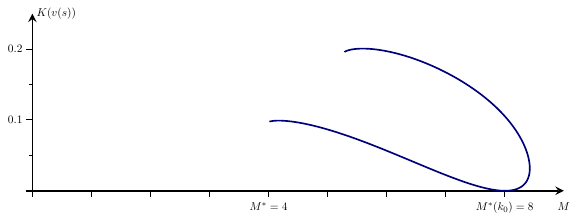}

        \includegraphics[width=0.7\textwidth]{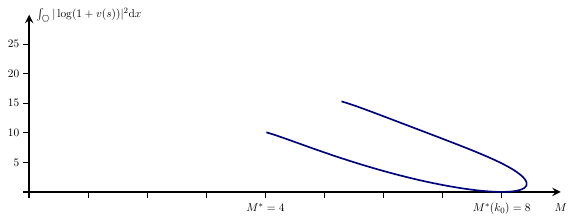}
        \caption{Plots of the minimal value, the constraint parameter $K(v(s))$, and the $\Lrm^2$-norm of the dominant logarithmic term for the up-hexagon and the down-hexagon branch. Note that the up-hexagon branch is distinguished by containing a fold.}
        \label{fig:hex plots}
	\end{figure}

    Similarly to the square case, we also performed a numerical continuation of the secondary bifurcation points found by \textsc{pde2path}. However, we did not find solutions which respect the symmetry assumptions. Therefore, it seems that no admissible secondary bifurcation occurs. Similarly, we considered the numerical continuation from the bifurcation point $(v,M^*(2k_0)) = (0,20)$ on a larger domain, which also did not yield admissible secondary bifurcations.

	\section{Discussion}\label{sec:discussion}

    In this paper, we establish the global bifurcation of stationary square and hexagonal solutions to the thermocapillary thin-film equation \eqref{eq:thin-film} from the pure conduction state using analytic global bifurcation theory. We show that the solutions exhibit a nodal structure, which prevents a return to the pure conduction state. Under the additional assumption that the Marangoni number is uniformly bounded along the bifurcation branch, we also establish that the minimal surface height tends to zero along a subsequence. Under this assumption, this shows that the thin-film equation \eqref{eq:thin-film} exhibits a family of stationary patterns which are close to film rupture.

    We conclude by giving a brief overview of related questions and discussing open problems.

    \paragraph{A rigorous bound on the Marangoni number}
    
    In the one-dimensional problem, one can establish a rigorous bound on the Marangoni number $M(s)$ along the bifurcation branch, see \cite{bruell2024}. The proof is based on a Hamiltonian structure of the spatial dynamics formulation of the one-dimensional version of \eqref{eq:integrated-stationary-problem} and a phase-plane analysis. Using this structure, one can show that for sufficiently large $M$, only strongly oscillating periodic solutions can appear. 
    
    Note that in two dimensions, an infinite-dimensional Hamiltonian structure for the spatial dynamics problem can be obtained from the elliptic problem \eqref{eq:integrated-stationary-problem} on an infinite cylinder in $x$ using periodicity in the transverse direction. However, since the phase space is infinite-dimensional, the one-dimensional argument does not generalise directly. Although the numerical continuation presented in \cref{sec:numerics} indicates that $M(s)$ is also bounded in the two-dimensional case, a rigorous proof remains an open question.

    \paragraph{Existence of a film-rupture solution}

    Using the uniform bound on the Marangoni number together with a uniform $\Hrm_\mathrm{per}^1(-\pi/k_0,\pi/k_0)$-bound, in the one-dimensional case it is possible to conclude the existence of a weak stationary periodic film-rupture solution by passing to a weak limit along a subsequence of the bifurcation branch. A key component of this argument uses the fact that $\Hrm_\per^1(-\pi/k_0,\pi/k_0)$ is embedded in $\Crm^{\frac12}(-\pi/k_0,\pi/k_0)$ and therefore convergence in $\Hrm^1_\per(-\pi/k_0,\pi/k_0)$ implies uniform convergence.
    Since this embedding does not hold in two dimensions, not even conditional on the uniform bounds on $M(s)$ and $K(v(s))$, the existence result for film-rupture solutions does not carry over into the two-dimensional case. We note that a uniform $\Hrm_\mathrm{per}^1(\Omega)$-bound can still be obtained from the weak formulation of the elliptic problem under these assumptions. Indeed, testing the equation with $1+v(s)$, it holds
    \begin{equation}\label{eq:conditional-Hrm-1-bound}
    \begin{split}
        \int_{\Omega} |\nabla v(s)|^2 \dd x & \leq \int_{\Omega} M(s) \log(1+v(s)) (1+v(s))  \dd x \\
        & \quad + \int_{\Omega} M(s) \sprlr{\frac{1}{2+v(s)} - \log(2+v(s)) + K(v(s))} (1+v(s)) \dd x 
    \end{split}
    \end{equation}
    Now, observe that by \Cref{bound v by M} and under the assumption that $M(s)$ and $K(v(s))$ are uniformly bounded, there is a uniform $\Lrm^\infty(\Omega)$-bound on $v(s)$. Hence, the right-hand side of \Cref{eq:conditional-Hrm-1-bound} is uniformly bounded. However, since $\Hrm_\per^1(\Omega)$ does not embed in $\Crm_\per^0(\Omega)$ in two dimensions, it is not possible to prove a sufficiently strong convergence result to pass to the limit in the weak formulation. In fact, a uniform bound on $v(s)$ in $\Wrm_\per^{1,p}(\Omega)$ for $p>2$ would be sufficient to establish a uniform $\Lrm_\per^q(\Omega)$-bound for $\log(1+v(s))$, $q>1$, which yields a uniform bound on $v(s)$ in $\Wrm_\per^{2,q}(\Omega)$ and the existence of a weak stationary periodic film-rupture solution similar to the one-dimensional case in \cite{bruell2024}. Although the rigorous analysis is open, we point out that the numerical analysis in \Cref{sec:numerics} suggests that $\log(1+v(s))$ remains bounded in $\Lrm_\per^2(\Omega)$, see \Cref{fig:squ plots,fig:hex plots}.

    \paragraph{Stability of large amplitude solutions}

    In \cref{sec:spec-stability,sec:stability-coperiodic}, we discuss the stability of the bifurcating periodic solutions close to the bifurcation point. In the one-dimensional case, one can even establish linear and energy instability of any positive, stationary periodic solution, see \cite{laugesen2000,laugesen2002}. Their approach is based on an idea introduced in \cite{bates1990} by transforming the linear operator into a self-adjoint operator, which then allows the characterisation of the smallest eigenvalue via Rayleigh quotients. However, their argument does not seem to generalise to the two-dimensional setting discussed in this paper. This is due to the quasilinear nature of equation \eqref{eq:thin-film} and therefore, the approach discussed in \cite{bates1990} for the multidimensional Cahn–Hilliard equation does not carry over.
	
	\section*{Acknowledgements}

    B.H.~was partially supported by the Swedish Research Council -- grant no.~2020-00440 -- and the Deutsche Forschungsgemeinschaft (DFG, German Research Foundation) -- Project-IDs 444753754 and 543917644. Additionally, B.H. acknowledges discussions about the numerical continuation in \cref{sec:numerics} with Dan J.~Hill.
    
    S.B.~has received funding from the Swedish Research Council -- grant no. 2020-00440.
    
    All three authors were supported by the Swedish Research Council under grant no. 2021-06594 while in residence at Institut Mittag-Leffler in Djursholm, Sweden during the fall semester of 2023.

    Finally, the authors also thank Erik Wahlén for interesting discussions on the global bifurcation problem.
	
	\section*{Data Availability Statement}

    The data generated for the numerical plots of the global bifurcation analysis in \cref{sec:intro,sec:numerics} is obtained using the \textsc{pde2path} library, which can be found on \cite{dohnal}. The code used to generate the corresponding data is available under \cite{bohmer2025}.

	\printbibliography
	
\end{document}